\documentclass[a4paper,reqno]{amsart}
\usepackage[english]{babel}
\usepackage{amssymb,amsmath,hyperref}
\usepackage{bbold}
\usepackage{amsrefs}
\usepackage{bussproofs, epigraph}
\usepackage{stackrel}

\newtheorem{thm}{Theorem}

\newtheorem{cor}[thm]{Corollary}
\newtheorem{defi}[thm]{Definition}
\newtheorem{rem}[thm]{Remark}
\newtheorem{nota}[thm]{Notation}

\newtheorem{princ}[thm]{Principle}

\newtheorem{theme}[thm]{Theme}

\numberwithin{thm}{section}

\newcommand\be{\begin{equation}}
\newcommand\ee{\end{equation}}

\newbox\gnBoxA
\newdimen\gnCornerHgt
\setbox\gnBoxA=\hbox{$\ulcorner$}
\global\gnCornerHgt=\ht\gnBoxA
\newdimen\gnArgHgt

\def\Godelnum #1{%
	\setbox\gnBoxA=\hbox{$#1$}%
	\gnArgHgt=\ht\gnBoxA%
	\ifnum \gnArgHgt<\gnCornerHgt
		\gnArgHgt=0pt%
	\else
		\advance \gnArgHgt by -\gnCornerHgt%
	\fi
	\raise\gnArgHgt\hbox{$\ulcorner$} \box\gnBoxA %
		\raise\gnArgHgt\hbox{$\urcorner$}}

\def\bdefi{\begin{defi}\rm}
\def\edefi{\end{defi}}
\def\bnota{\begin{nota}\rm}
\def\enota{\end{nota}}
\def\brem{\begin{rem}\rm}
\def\erem{\end{rem}}

\def\QFAC{\textup{QF-AC}}
\def\QFACP{\big(\textup{QF-AC}^{1,0}\big)^{\st}}

\def\RCA{\textup{RCA}}

\def\WKL{\textup{WKL}}

\def\bye{\end{document}}

\def\N{{\mathbb  N}}

\def\R{{\mathbb  R}}

\def\FAN{\textup{FAN}}

\def\R{{\mathbb{R}}}
\def\({\textup{(}}
\def\){\textup{)}}

\def\st{\textup{st}}
\def\asa{\leftrightarrow}

\def\di{\rightarrow}

\def\eps{\varepsilon}

\def\US{\textit{US}}

\def\M{\mathcal{M}}

\def\ACA{\textup{ACA}}

\def\paai{\Pi_{1}^{0}\textup{-TRANS}}

\def\Ph{\mathbb{ph}}
\def\Ps{\mathbb{ps}}

\newbox\gnBoxA
\newdimen\gnCornerHgt
\setbox\gnBoxA=\hbox{$\ulcorner$}
\global\gnCornerHgt=\ht\gnBoxA
\newdimen\gnArgHgt

\def\bdefi{\begin{defi}\rm}
\def\edefi{\end{defi}}
\def\bnota{\begin{nota}\rm}
\def\enota{\end{nota}}
\def\brem{\begin{rem}\rm}
\def\erem{\end{rem}}

\def\rel{\sqsubseteq}

\def\sler{\sqsupset}

\def\RCA{\textup{RCA}}

\def\RCAo{\textup{RCA}_{0}^{\omega}}
\def\RCAO{\textup{RCA}_{0}^{\Omega}}

\def\WKL{\textup{WKL}}

\def\bye{\end{document}}

\def\N{{\mathbb  N}}

\def\R{{\mathbb  R}}

\def\FAN{\textup{FAN}}
\def\UFAN{\textup{UFAN}}

\def\R{{\mathbb{R}}}
\def\({\textup{(}}
\def\){\textup{)}}

\def\asa{\leftrightarrow}

\def\di{\rightarrow}

\def\eps{\varepsilon}

\def\ACA{\textup{ACA}}

\numberwithin{equation}{section}

\begin{document}
\title[Reverse Mathematics of Brouwer's Continuity Theorem]{Reverse Mathematics of Brouwer's Continuity Theorem and related principles}
\author{Sam Sanders
}



\begin{abstract}
In intuitionistic mathematics, the \emph{Brouwer Continuity Theorem} states that all total real functions are (uniformly) continuous on the unit interval.  
We study this theorem and related principles from the point of view of Reverse Mathematics over a base theory accommodating higher types and Nonstandard Analysis.
With regard to the bigger picture, Reverse Mathematics provides a classification of theorems of ordinary mathematics 
based on computability.  Our aim is to provide an \emph{alternative} classification of theorems based on the central tenet of Feferman's \emph{Explicit Mathematics}, namely that \emph{a proof of existence of an object yields a procedure to compute 
said object}.  Our classification gives rise to the \emph{Explicit Mathematics theme} (EMT).   Intuitively speaking, the EMT states that
a standard object with certain properties can be computed by a functional if and only if this object exists classically with these same standard \emph{and nonstandard} properties.  
Hence, we establish the EMT for a series of intuitionistic principles in this paper.        
\end{abstract}

\maketitle


\section{Introduction: Intuitionistic, Explicit, and Reverse Mathematics}
\subsection{Intuitionistic and Reverse Mathematics}
At the beginning of the twentieth century, L.E.J.\ Brouwer proposed \emph{intuitionism}, an anti-platonist philosophy of mathematics (\cite{brouw}).  
Brouwer was motivated by the belief that mathematics is the result of human mental activity, not the discovery of pre-existent entities in some independent reality.  His philosophical ideas led him to reject the principle of excluded middle as a valid logical law (\cite{vajuju}*{p.\ 334}). 
Brouwer also initiated the development of \emph{intuitionistic mathematics}, a type of constructive mathematics motivated by his ideas and seemingly incompatible with mainstream (or `classical') mathematics.  
In particular, Brouwer's \emph{Continuity Theorem} (\cite{vajuju}*{Theorem 3, p.~463}) states that all total $[0,1]\di \R$ functions are (uniformly) continuous.               

\medskip  

Recent results in Reverse Mathematics (\cite{aloneatlast3}*{Theorem 43}) consider the relation between nonstandard continuity\footnote{By `nonstandard continuity', we mean the definition from 
Nonstandard Analysis involving the `infinitely close' predicate `$\approx$' as in \cite{stroyan}*{\S5.1, p.\ 71}.  See Section \ref{norema} below for the exact definition.} and $\eps$-$\delta$-continuity.  
A natural question is then:  
\begin{center}
\emph{In case all total $[0,1]\di \R$-functions are $\eps$-$\delta$-continuous, are they then also \textbf{nonstandard} continuous?  What extra nonstandard axioms are needed?}  (Q).
\end{center}  
The short answer to (Q) is \emph{Yes, see Theorem~\ref{allisone}}.
The long answer takes up the rest of this paper: We shall develop Reverse Mathematics (RM for short; See Section~\ref{theme} for the latter) for Brouwer's continuity theorem and related principles over a conservative extension of the `usual' base theory $\RCA_{0}$ involving higher types and Nonstandard Analysis.  This extended base theory, called $\RCAO$, is based on Nelson's \emph{internal set theory} (\cite{wownelly}), as discussed in Section \ref{pdef}.  
The aforementioned development of RM takes place in Section~\ref{main}-\ref{strongEMT} and proceeds along the lines of the \emph{Explicit Mathematics Theme} (EMT for short), discussed in the next section.  

\medskip

As an aside, our study gives rise to several very natural splitting results (See Section~\ref{fans}) \emph{and} the discovery of a \emph{natural higher-order} statement implicit in a second-order theorem concerning continuous functions (See Section \ref{LOCO}).  This implicit presence is caused by the special nature of the RM-definition of continuity.  
\subsection{The theme from Explicit Mathematics}\label{theme}
Reverse Mathematics is a program in the foundations of mathematics initiated by Friedman (\cite{fried, fried2}), and developed extensively by Simpson and others (See \cite{simpson1, simpson2} for an overview and introduction).  
The aim is to find the axioms necessary to prove a given theorem of ordinary$^{\ref{footsie}}$ mathematics, assuming the `base theory' $\RCA_{0}$, a weak system of computable mathematics.  

\medskip

In particular, RM can be viewed as a classification of theorems of ordinary\footnote{The term `ordinary mathematics' refers to mathematics concerned with countable and separable objects, as discussed in \cite{simpson2}*{I.1}.\label{footsie}} mathematics from the point of view of \emph{computability} (See e.g.\ \cite{simpson2}*{I.3.4}).  A natural question is if there are \emph{other interesting ways} of classifying these theorems;  
In this paper, we shall discuss a classification based on the core tenet of Feferman's \emph{Explicit Mathematics}  (See \cite{feferman2,feferman3,feferman4, fefermaninf} and \cite{firstHORM}*{\S1.3}), which is as follows: 
\begin{center}
\emph{a proof of existence of an object yields a procedure to compute said object}.   
\end{center}
Hence, rather than enforcing the core tenet of Explicit Mathematics, we shall classify theorems based on `how much' extra is needed to compute objects claimed to exist by theorems of ordinary mathematics.  
This classification will be developed along the lines of the following general theme, first introduced in \cite{firstHORM}.  
\begin{theme}[The theme from Explicit Mathematics]\rm
Consider a theorem of mathematics (in the language of $\RCAO$) of the form:
\[
T^{\st}\equiv(\forall^{\st}x^{\sigma})(A^{\st}(x)\di (\exists^{\st} y^{\tau})B^{\st}(x,y)).
\]
The \emph{nonstandard} version of $T^{\st}$ is the statement:
\be\label{mf}\tag{$T^{*}$}
(\forall^{\st}x^{\sigma})(A^{\st}(x)\di (\exists^{\st} y^{\tau})B(x,y)), 
\ee
where $B^{\st}$ is `transferred' to $B$, i.e.\ the standardness predicate `st' is omitted.  
Furthermore, the \emph{uniform} version of $T$, is 
\be\tag{$UT$}
(\exists \Phi^{\sigma\di \tau})(\forall x^{\sigma})(A (x)\di B (x,\Phi(x))).
\ee
The \emph{Explicit Mathematics Theme} (EMT) is the observation that for many theorems $T$ as above, the base theory proves $T^{*}\asa UT$.  
\end{theme}
As suggested by its name, the EMT is inspired by the foundational program {Explicit Mathematics}.  
The name `EMT' was chosen because it expresses a uniform way of characterising the computability from the central tenet of Explicit Mathematics, namely that the mere non-uniform \emph{existence} of an object $y$ as in $T^{*}$, is 
equivalent to $y$ \emph{being computable via a functional} as in $UT$.  

\medskip

In this paper, we will establish EMT for a number of intuitionistic principles.  
In light of \cite{kohlenbach2}*{p.\ 293-294}, the \emph{fan functional} constitutes a natural starting point, discussed in Section \ref{main}.  
An obvious next step is the study, in Section \ref{UB41}, of the uniform boundedness principles from \cite{kohlenbach4}*{Ch.\ 12}, which are generalisations of the fan functional more suitable for proof mining.  
In turn, in Section \ref{SB}, we study continuity principles which are weaker than the fan functional.  
Finally, in Section~\ref{strongEMT}, we discuss the RM-classification of Brouwer's continuity theorem.  The latter study gives rise to very natural splitting results, as discussed in Section~\ref{fans}.         

\medskip

While studying principles weaker than the fan functional in Section \ref{SB}, it becomes clear that the \emph{idealization} axiom I from $\RCAO$ is needed (which is exceptional in this context).  
Furthermore, the axiom $\textup{I}$ gives rise to \emph{another} nonstandard version $T^{**}$, also equivalent to $UT$ as discussed in Section \ref{layola}.   

\medskip

In conclusion, we discuss the possible foundational significance of the EMT.  
\begin{enumerate} 
\item Central to the EMT is that statements involving \emph{higher-type} objects like $UT$ are equivalent to statements $T^{*}$ involving only \emph{lower-type} nonstandard objects.  
In this light, it seems incoherent to claim that higher-type objects are somehow `more real' than nonstandard ones (or vice versa).  
Furthermore, the EMT suggests that higher-order RM is actually implicit in Friedman-Simpson RM, as Nonstandard Analysis is used in the latter:  See \cites{tahaar, tanaka1, keisler1, yo1, yokoyama2, yokoyama3, sayo, avi3, aloneatlast3}.  
More directly, the EMT even yields an example of a uniform statement implicit in a second-order statement concerning continuous functions (See Remark~\ref{loofer}).        
\item In general, to prove $T^{*}\di UT$, one defines a functional $\Psi(\cdot, M)$ of (rather) elementary complexity, but involving an infinite number $M$.  
Assuming $T^{*}$, this functional is $\Omega$-invariant (See Definition \ref{homega}) and the axiom $\Omega$-CA from $\RCAO$ provides the required standard functional for $UT$.  The functional $\Psi(\cdot, M)$ is the \emph{canonical approximation} of the one from $UT$.  
As discussed in \cite{firstHORM}, these results can be viewed as a contribution to Hilbert's program for finitistic mathematics, as infinitary objects (the functional from $UT$) are decomposed into elementary computable objects.    
By the results in the next sections and in \cite{firstHORM}, such decomposition is available for both classical and intuitionistic principles, i.e.\ a `finitistic multiverse' presents itself.  
\item Fujiwara and Kohlenbach have established the connection (and even equivalence in some cases) between (classical) uniform existence as in $UT$ and intuitionistic provability (\cites{fuji1,fuji2}).  The EMT suggests that $T^{*}$ constitutes another way of capturing intuitionistic provability (in certain cases).  
\end{enumerate}
Finally, we urge the reader to first consult Remarks \ref{ohdennenboom} and \ref{flack} so as to clear up any common prejudice regarding Nelson's framework.

\section{A base theory for Reverse Mathematics}\label{pdef}
In this section, we introduce the base theory $\RCAO$ in which we will work.  
We discuss some basic results and introduce some notation.
\subsection{The system $\RCAO$}
In two words, $\RCAO$ is a conservative extension of Kohlenbach's base theory $\RCAo$ from \cite{kohlenbach2} with certain axioms from Nelson's \emph{Internal Set Theory} (\cite{wownelly}) based on the approach from \cites{brie,bennosam}.    
This conservation result is proved in \cite{bennosam}, while partial results are implicit in \cite{brie}.  The system $\RCAo$ is a conservative extension of $\RCA_{0}$ for the second-order language by \cite{kohlenbach2}*{Prop.\ 3.1}.  

\medskip

In Nelson's \emph{syntactic} approach to Nonstandard Analysis (\cite{wownelly}), as opposed to Robinson's semantic one (\cite{robinson1}), a new predicate `st($x$)', read as `$x$ is standard' is added to the language of ZFC.  
The notations $(\forall^{\st}x)$ and $(\exists^{\st}y)$ are short for $(\forall x)(\st(x)\di \dots)$ and $(\exists y)(\st(y)\wedge \dots)$.
The three axioms \emph{Idealization}, \emph{Standard Part}, and \emph{Transfer} govern the new predicate `st'  and give rise to a conservative extension of ZFC.   
Nelson's approach has been studied in the context of higher-type arithmetic in e.g.\ \cite{brie, bennosam, avi3}.

\medskip

Following Nelson's approach in arithmetic, we define $\RCAO$ as the system 
\be\label{poli}
\textup{E-PRA}_{\st}^{\omega*}+\textup{QF-AC}^{1,0} +\textup{I}+\textup{HAC}_{\textup{int}}+ \textup{PF-TP}_{\forall} 
\ee
from \cite{bennosam}*{\S3.2-3.3}.  Nelson's idealization axiom I is available in $\RCAO$, but to guarantee that the latter is a conservative extension of $\RCAo$, Nelson's axiom \emph{Standard part} must be limited to $\Omega$-CA defined below (which derives from HAC$_{\textup{int}}$), 
while Nelson's axiom \emph{Transfer} has to be limited to universal formulas \emph{without} parameters, as in PF-TP$_{\forall}$.  We have the following theorem.  
\begin{thm}
The system \textup{E-PRA$_{\st}^{\omega*}+\textup{HAC}_{\textup{int}}+\textup{I}+ \textup{PF-TP}_{\forall}$} is a conservative extension of \textup{E-PRA}$^{\omega}$.  
The system $\RCAO$ is a $\Pi_{2}^{0}$-conservative extension of $\textup{PRA}$.  
\end{thm}
\begin{proof}
See \cite{bennosam}*{Cor.\ 9}.  
\end{proof}
The conservation result for $\textup{E-PRA}_{\st}^{\omega*}+\textup{QF-AC}^{1,0}$ is trivial.  
Furthermore, omitting PF-TP$_{\forall}$, the theorem is implicit in \cite{brie}*{Cor.\ 7.6} as the proof of the latter goes through as long as EFA is available.
We now discuss the two final axioms of \eqref{poli}.   
\subsection{Transfer and Standard Part in $\RCAO$}
We first discuss the \emph{Transfer principle} included in $\RCAO$, which is as follows.     
\begin{princ}[PF-TP$_\forall$]  
For any internal formula $\varphi(x^{\tau})$ with all parameters shown, we have $(\forall^{\st}x^{\tau})\varphi(x)\di (\forall x)\varphi(x)$.
\end{princ} 
A special case of the previous can be found in Avigad's system NPRA$^{\omega}$ from \cite{avi3}.  
The omission of parameters in PF-TP$_{\forall}$ is essential, as is clear from the following theorem, relating to the following principles:
\be\tag{$\paai$}
 (\forall^{\st}f^{1})\big[(\forall^{\st}n^{0})f(n)=_{0}0\di (\forall n^{0})f(n)=_{0}0],
\ee 
\be\tag{$\exists^{2}$}
(\exists \varphi^{2})(\forall g^{1})\big[(\exists x^{0})g(x)=0 \asa \varphi(g)=0  \big].
\ee
Note that standard parameters are allowed in $f$, and that $(\exists^{2})$ is the functional version of $\ACA_{0}$ (\cite{simpson2}*{III}), i.e.\ arithmetical comprehension.
\begin{thm}\label{markje}
The system $\RCAO$ proves $\paai\asa (\exists^{2})$.
\end{thm}
\begin{proof}
By \cite{bennosam}*{Corollary 12}.
\end{proof}
Next, we discuss the \emph{Standard Part principle}, called $\Omega$-CA, included in $\RCAO$.  
Intuitively speaking, a Standard Part principle allows us to convert nonstandard into standard objects.  
By way of example, the following type 1-version of the Standard part principle results in a conservative extension of $\WKL_{0}$ (See \cites{keisler1, briebenno}).  
\be\label{STP}\tag{STP}
(\forall X^{1})(\exists^{\st} Y^{1})(\forall^{\st} x^{0})(x\in X\asa x\in Y).
\ee
Here, we have used set notation to increase readability;  We assume that sets $X^{1}$ are given by their characteristic functions $f^{1}_{X}$, i.e.\ $(\forall x^{0})[x\in X\asa f_{X}(x)=1]$.     
The set $Y$ from \eqref{STP} is also called the \emph{standard part} of $X$.  We also write `$N^{0}\in \Omega$' as short for `$\neg\st(N)$' and say that `$N$ is infinite'.   

\medskip
    
We now discuss the Standard Part principle $\Omega$-CA, a very practical consequence of the axiom HAC$_{\textup{int}}$.  
Intuitively speaking, $\Omega$-CA expresses that we can obtain 
the standard part (in casu $G$) of \emph{$\Omega$-invariant} nonstandard objects (in casu $F(x,M)$).   
\bdefi[$\Omega$-invariance]\label{homega} Let $F^{(\sigma\times  0)\di 0}$ be standard and fix $M^{0}\in \Omega$.  
Then $F(\cdot,M)$ is {\bf $\Omega$-invariant} if   
\be\label{homegainv}
(\forall^{\st} x^{\sigma})(\forall N^{0}\in \Omega)\big[F(x ,M)=_{0}F(x,N) \big].  
\ee
\edefi
\begin{princ}[$\Omega$-CA]\rm Let $F^{(\sigma\times 0)\di 0}$ be standard and fix $M^{0}\in \Omega$.
For every $\Omega$-invariant $F(\cdot,M)$, there is a standard $G^{\sigma\di 0}$ such that
\be\label{homegaca}
(\forall^{\st} x^{\sigma})(\forall N^{0}\in \Omega)\big[G(x)=_{0}F(x,N) \big].  
\ee
\end{princ}
The axiom $\Omega$-CA provides the standard part of a nonstandard object, if the latter is \emph{independent of the choice of infinite number} used in its definition.  
Proofs may be found in \cite{tale} or \cite{firstHORM}.
\begin{thm}\label{drifh}
In the system $\RCAO$, the principle $\Omega\textup{-CA}$ is provable.  
\end{thm}
\begin{cor}\label{genall}
In $\RCAO$, we have for all standard $F^{(\sigma\times 0)\di 1}$ that
\begin{align*}
(\forall^{\st} x^{\sigma})(\forall M,N \in & \Omega)\big[F(x ,M)\approx_{1}F(x,N) \big] \\
&\di (\exists^{\st}G^{\sigma\di 1})(\forall^{\st} x^{\sigma})(\forall N^{0}\in \Omega)\big[G(x)\approx_{1}F(x,N) \big],
\end{align*}
where $f^{1}\approx_{1} g^{1}$ if $(\forall^{\st}n^{0})(f(n)=_{0}g(n))$. 
\end{cor}
\begin{cor}\label{genalli}
In $\RCAO$, for all standard $F^{(\sigma\times 0)\di 1}$ and internal formulas $C$,
\begin{align*}
(\forall^{\st} x^{\sigma})(\forall M,N \in & \Omega)\big[C(F,x)\di F(x ,M)\approx_{1}F(x,N) \big] \\
&\di (\exists^{\st}G^{\sigma\di 1})(\forall^{\st} x^{\sigma})(\forall N^{0}\in \Omega)\big[C(F,x)\di  G(x)\approx_{1}F(x,N) \big].
\end{align*}
\end{cor}
Applications of the previous corollaries are assumed to be captured under the umbrella-term `$\Omega$-CA'.  
Furthermore, by the above, if we drop the $\Omega$-invariance condition in $\Omega$-CA, the resulting system is a non-conservative extension of $\RCAO$.       
 
\subsection{Notations and remarks}\label{norema}
We finish this section with some remarks and notations regaring $\RCAO$.  
First of all, we shall mostly use notations as in \cite{bennosam}.  
\begin{rem}[Notations]\label{notawin}\rm
We write $(\forall^{\st}x^{\tau})\Phi(x^{\tau})$ and $(\exists^{\st}x^{\sigma})\Psi(x^{\sigma})$ as short for 
$(\forall x^{\tau})\big[\st(x^{\tau})\di \Phi(x^{\tau})\big]$ and $(\exists^{\st}x^{\sigma})\big[\st(x^{\sigma})\wedge \Psi(x^{\sigma})\big]$.     
We also write $(\forall x^{0}\in \Omega)\Phi(x^{0})$ and $(\exists x^{0}\in \Omega)\Psi(x^{0})$ as short for 
$(\forall x^{0})\big[\neg\st(x^{0})\di \Phi(x^{0})\big]$ and $(\exists x^{0})\big[\neg\st(x^{0})\wedge \Psi(x^{0})\big]$.  Furthermore, if $\neg\st(x^{0})$ (resp.\ $\st(x^{0})$), we also say that $x^{0}$ is `infinite' (resp.\ finite) and write `$x^{0}\in \Omega$'.  
Finally, a formula $A$ is `internal' if it does not involve $\st$, and $A^{\st}$ is defined from $A$ by appending `st' to all quantifiers (except bounded number quantifiers).    
\end{rem}
Secondly, we use the usual notations for rational and real numbers and functions as introduced in \cite{kohlenbach2}*{p.\ 288-289} (and \cite{simpson2}*{I.8.1} for the former).  
\begin{nota}[Real number]\label{keepintireal}\rm
A (standard) real number $x$ is a (standard) fast-converging Cauchy sequence $q_{(\cdot)}^{1}$, i.e.\ $(\forall n^{0}, i^{0})(|q_{n}-q_{n+i})|<_{0} \frac{1}{2^{n}})$.  
We freely make use of Kohlenbach's `hat function' from \cite{kohlenbach2}*{p.\ 289} to guarantee that every sequence $f^{1}$ can be viewed as a real.  
Two reals $x, y$ represented by $q_{(\cdot)}$ and $r_{(\cdot)}$ are \emph{equal}, denoted $x=y$, if $(\forall n)(|q_{n}-r_{n}|\leq \frac{1}{2^{n}})$. Inequality $<$ is defined similarly.         
We also write $x\approx y$ if $(\forall^{\st} n)(|q_{n}-r_{n}|\leq \frac{1}{2^{n}})$ and $x\gg y$ if $x>y\wedge x\not\approx y$.  Functions $F:\R\di \R$ mapping reals to reals are represented by functionals $\Phi^{1\di 1}$ such that $(\forall x, y)(x=y\di \Phi(x)=\Phi(y))$, i.e.\ equal reals are mapped to equal reals.   
\end{nota}
Thirdly, by way of context for the next remark, recall that extending the language of a logical system with symbols representing certain functionals is common practice in mathematical logic: Indeed, see e.g.\ \cite{farwise}*{p.\ 935, \S4.5}, \cite{avi2}*{\S2.5} and \cites{fefja1,fefja2}.
\begin{rem}[Standard functionals]\label{tokkiep}\rm
We discuss some consequences of {PF-TP$_{\forall}$};  In particular, how the latter gives rise to \emph{standard and unique} functionals.    

\medskip

First of all, consider the fan functional, defined as follows:
\be\label{MUC}\tag{MUC}
(\exists \Omega^{3})\big[(\forall \varphi^{2}, f^{1}, g^{1}\leq_{1}1 )[\overline{f}(\Omega(\varphi))=_{0}\overline{g}(\Omega(\varphi))\di \varphi(f)=_{0}\varphi(g)]\big].
\ee  
We immediately obtain, via the contraposition of PF-TP$_{\forall}$, that 
\be\label{drifsd}
(\exists^{\st} \Theta^{3})(\forall \varphi^{2}, f^{1}, g^{1}\leq_{1}1 )[\overline{f}(\Theta(\varphi))=_{0}\overline{g}(\Theta(\varphi))\di \varphi(f)=_{0}\varphi(g)],
\ee
as the formula in big square brackets in \eqref{MUC} is internal and does not have parameters other than $\Omega$.  
In other words, we may assume that the fan functional is \emph{standard} and the same holds for \emph{any functional} of which the (internal) definition does not involve additional parameters.   

\medskip

Secondly, again for the fan functional, we may assume $\Omega(\varphi)$ is the \emph{least number} as in \eqref{MUC}, which implies that $\Theta(\varphi)$ from \eqref{drifsd} can also be assumed to have this property.  
However, then $\Theta(\varphi)=_{0}\Omega(\varphi)$ for any $\varphi^{2}$, implying $\Theta=_{3}\Omega$, i.e.\ if it exists, the fan functional is \emph{unique and standard}.  
The same again holds for any uniquely-defined functional of which the internal definition is parameter-free.  

\medskip

The two above observations prompted the addition to $\RCAO$ of axioms reflecting the uniqueness and standardness of certain functionals (See \cite{bennosam}*{\S3.3}).  
In particular, the language of $\RCAO$ contains a distinct symbol $\Omega_{0}$ and the system itself contains:
\be\label{durfall}
\st(\Omega_{0})\wedge (\forall^{\st}\Xi^{3})\big[M^{\st}(\Xi)\di (\forall^{\st} \varphi^{2})(\Omega_{0}(\varphi)=_{0}\Xi(\varphi))\big], 
\ee
where $M(\Omega)$ is the formula in square brackets in \eqref{MUC}, with the addition that $\Omega(\varphi)$ is the least number with this property.  

\medskip

Clearly, the axiom \eqref{durfall} expresses that, if it exists, the fan functional is standard and unique, reflecting the standardness and uniqueness properties we have proved in the previous two paragraphs assuming \eqref{MUC}.  
Furthermore, as noted in \cite{bennosam}*{\S3.3}, $\RCAO$ contains axioms like \eqref{durfall} for uniquely defined (via an internal formula) functionals.
An advantage of \eqref{durfall} is that $\RCAO$ proves that $\eqref{MUC}^{\st}\di \eqref{MUC}$ by applying PF-TP$_{\forall}$ to $M^{\st}(\Omega_{0})$, as discussed in \cite{bennosam}*{\S3.3} and \cite{firstHORM}*{\S4}.  
We stress that \eqref{durfall} does not represent some `trick' to obtain equivalences: This formula reflects the standard and unique nature of the fan functional which we proved above.  
%
%
\end{rem}
Fourth, we show that versions of \eqref{durfall}, and the associated equivalences, can also be obtained \emph{without} invoking the uniqueness of the functional at hand.  
\begin{rem}[Standard functionals II]\label{tokkier}\rm
We discuss important consequences of PF-TP$_{\forall}$;  In particular how the latter gives rise to basic \emph{standard} properties of functionals.    
By way of example, consider the modulus-of-continuity functional:
\be\label{POC2}\tag{MPC}
(\exists \Delta^{3})\big[(\forall \varphi^{2}, f^{1}, g^{1}\leq_{1}1 )(\overline{f}\Delta(\varphi,f)=_{0}\overline{g}\Delta(\varphi,f)\di \varphi(f)=_{0}\varphi(g))\big].
\ee
Kohlenbach shows in \cite{kohlenbach4}*{\S4} that an \emph{associate} (See \cite{kohlenbach4}*{Def.\ 4.2} or Definition~\ref{kodef} below) can be defined from a modulus of continuity.  
Thus, let $\Xi(\Phi, \omega_{\Phi})$ be the functional $\alpha$ from the second part of the proof of \cite{kohlenbach4}*{Prop.\ 4.4} which produces an associate for $\Phi^{2}$ from the latter and a modulus of continuity $\omega_{\Phi}$ of $\Phi$.  

\medskip

Working in $\RCAO+\eqref{POC2}$, both $\Xi$ and the functional $\Delta$ from \eqref{POC2} are standard, and it is clear that the standard functional $\Xi(\varphi, \Delta(\varphi, \cdot))$ produces a standard associate for any standard $\varphi^{2}$.  
By the definition of associate and the fact that $\Delta$ is standard, we have the following \emph{standard} property:
\be\label{torque}
(\forall^{\st} \varphi^{2}, f^{1}\leq_{1}1)(\exists^{\st} n^{0})\big(\Xi(\varphi, \Delta(\varphi, \cdot))(\overline{f}n)>0\big).
\ee  
Applying $\QFAC^{2,0}$ relative to `st' (which follows from HAC$_{\textup{int}}$), there is a standard functional $\Psi^{3}$ witnessing $n$ in \eqref{torque}.  Again by the definition of associate: 
\be\label{dorlpppp}
(\forall^{\st} \varphi^{2}, f^{1}\leq_{1}1)\big[\Xi(\varphi, \Delta(\varphi, \cdot))(\overline{f}\Psi(\varphi,f))=_{0}\varphi(f)+1\big].
\ee
In short, \emph{if} there is a modulus-of-continuity functional as in \eqref{POC2}, \emph{then} we can obtain a \emph{standard} `associate functional' $\Xi$ and a suitable \emph{standard} modulus-of-continuity functional $\Psi$, which allow us to represent standard type two objects as countable ones as in \eqref{dorlpppp}.  The same observation goes through for \eqref{POC}$^{\st}$.    

\medskip

We now cast this observation into an axiom, namely the conjunct of $\st(\Lambda_{0})$ and:
\be\label{dorg}
(\forall^{\st}\Upsilon^{3})\big[K^{\st}(\Upsilon)\di (\forall^{\st}\varphi^{2}, f^{1}\leq_{1}1)[ \Xi(\varphi, \Upsilon(\varphi, \cdot))(\overline{f}\Lambda_{0}(\varphi, f))=\varphi(f)+1] \big],
\ee
where $K(\Delta)$ is the formula in square brackets in \eqref{POC2} and where $\Lambda_{0}^{3}$ is a new symbol added to the language of $\RCAO$.  
Any model $\M$ of $\RCAO$ can easily be extended to satisfy \eqref{dorg}: If there is standard $\Upsilon$ in $\M$ such that the latter satisfies $K^{\st}(\Upsilon)$, then \eqref{torque} holds in $\M$ for $\Delta$ replaced by $\Upsilon$. 
As a consequence, $\M$ contains $\Psi$ (standard in $\M$) such that \eqref{dorlpppp} holds in $\M$.  Now interpret $\Lambda_{0}$ as $\Psi$ in $\M$.  
In this light, we shall assume that $\RCAO$ has been extended with \eqref{dorg}.  

\medskip

We stress that \eqref{dorg} merely introduces a Skolem constant for a functional which can be derived from a (standard) modulus-of-continuity functional, assuming the latter exists, i.e.\ \eqref{dorg} merely formalises an observation made in the previous paragraphs.    
Furthermore, we show in Section \ref{LOCO} that \eqref{dorg} also allows us to prove $\eqref{POC}^{\st}\di \eqref{POC}$.
In conclusion, even without the use of uniqueness properties as in the previous remark, we can obtain useful versions of \eqref{durfall}.       
\end{rem}
Finally, one could view \eqref{durfall} and \eqref{dorg} as establishing basic properties of mathematical objects, which after all is one of the tasks of any base theory for RM.

\medskip

Fifth, we discuss  the notion of equality in $\RCAO$.  
\begin{rem}[Equality]\label{equ}\rm
The system $\RCAo$ only includes equality between natural numbers `$=_{0}$' as a primitive.  Equality `$=_{\tau}$' for type $\tau$-objects $x,y$ is defined as:
\be\label{aparth}
[x=_{\tau}y] \equiv (\forall z_{1}^{\tau_{1}}\dots z_{k}^{\tau_{k}})[xz_{1},\dots, z_{k}=_{0}yz_{1}\dots z_{k}]
\ee
if the type $\tau$ is composed as $\tau\equiv(\tau_{1}\di \dots\di \tau_{k}\di 0)$.
In the spirit of Nonstandard Analysis, we define `approximate equality $\approx_{\tau}$' as follows:
\be\label{aparth2}
[x\approx_{\tau}y] \equiv (\forall^{\st} z_{1}^{\tau_{1}},\dots, z_{k}^{\tau_{k}})[xz_{1}\dots z_{k}=_{0}yz_{1}\dots z_{k}]
\ee
with the type $\tau$ as above.  
The system $\RCAo$ includes the axiom of extensionality:
\be\label{EXT}\tag{E}  
(\forall \varphi^{\rho\di \tau})(\forall  x^{\rho},y^{\rho}) \big[x=_{\rho} y \di \varphi(x)=_{\tau}\varphi(y)   \big],
\ee
but as noted in \cite{brie}*{p.\ 1973}, the so-called axiom of standard extensionality \eqref{EXT}$^{\st}$ is problematic and cannot be included in $\RCAO$.  Nonetheless, instances of \eqref{EXT}$^{\st}$ can be obtained, as is clear from Theorem \ref{halleh}.
Furthermore, in light of Corollary \ref{genall}, it is obvious how $\Omega$-CA can be further generalised to $F^{(\sigma\times 0)\di \tau}$ using $\approx_{\tau}$ instead of $\approx_{1}$.  The same holds for `$\approx$' if $\tau=1$ and $F$ is a real-valued function.    
\end{rem}
Finally, we discuss the role of Tennenbaum's theorem in Nelson's framework.
\begin{rem}[The computable nature of operations in $\RCAO$]\label{ohdennenboom}\rm
Tennenbaum's theorem (\cite{kaye}*{\S11.3}) `literally' states that any nonstandard model of PA is not computable.  \emph{What is meant} is that for a nonstandard model $\M$ of PA, the operations $+_{\M}$ and $\times_{\M}$ cannot be computably defined in terms of the operations $+_{\N}$ and $\times_{\N}$ of the standard model $\N$ of PA.  

\medskip

While Tennenbaum's theorem is of interest to the \emph{semantic} approach to Nonstandard Analysis involving nonstandard models, $\RCAO$ is based on Nelson's \emph{syntactic} framework, and therefore Tennenbaum's theorem does not apply:  Any attempt at defining the (external) function `$+$ limited to the standard numbers' is an instance of \emph{illegal set formation}, forbidden in Nelson's \emph{internal} framework (\cite{wownelly}*{p.\ 1165}).  

\medskip

To be absolutely clear, lest we be misunderstood, Nelson's \emph{internal set theory} IST forbids the formation of \emph{external} sets $\{x\in A: \st(x)\}$ and functions `$f(x)$ limited to standard $x$'.  
Therefore, any appeal to Tennenbaum's theorem to claim the `non-computable' nature of $+$ and $\times$ from $\RCAO$ is blocked, for the simple reason that the functions `$+$ and $\times$ limited to the standard numbers' do not exist.              
On a related note, we recall Nelson's dictum from \cite[p.\ 1166]{wownelly} as follows:
\begin{quote}
\emph{Every specific object of conventional mathematics is a standard set.} It remains unchanged in the new theory \textup{[IST]}.  
\end{quote}
In other words, the operations `$+$' and `$\times$', but equally so primitive recursion, in (subsystems of) IST, are \emph{exactly the same} familiar operations we  
know from (subsystems of) ZFC.  Since the latter is a first-order system, we however cannot exclude the presence of nonstandard objects, and internal set theory just makes this explicit, i.e.\ IST turns a supposed bug into a feature.    
\end{rem}

\section{The EMT for the fan functional and related principles}\label{main}
In this section, we establish the EMT for principles related to the \emph{fan functional}.  
The latter was introduced by Tait as the first example of a functional which is \emph{non-obtainable}, i.e.\ not computable from lower-type objects (See \cite{noortje}*{p.\ 102}). 

\medskip

In intuitionistic mathematics, the fan functional emerges as follows:  By \cite{troelstra1}*{2.6.6, p.\ 141}, if a universe of functions $\mathfrak{U}$ satisfies $\bf{EL}+\FAN$, then the class ECF$(\mathfrak{U})$ of \emph{extensional continuous functionals relative to $\mathfrak{U}$}, contains a fan functional.  Here, $\bf{EL}$ is a basic system of intuitionistic mathematics and FAN is the fan theorem, the classical contraposition of WKL.  Similar results on the fan functional are in \cites{troelstra2,troelstra3,gandymahat}.    
In our notation, the (existence of the) fan functional is: 
\be\label{MUC3}\tag{MUC}
(\exists \Omega^{3})(\forall \varphi^{2}) (\forall g^{1}, f^{1}\leq_{1}1 )\big[\overline{g}\Omega(\varphi)=_{0}\overline{f}\Omega(\varphi)\di \varphi(g)=_{0}\varphi(f)\big].
\ee
By \cite{kohlenbach2}*{Prop.\ 3.13} and \cite{bennosam}*{Theorem 5}, the system $\RCAO+\eqref{MUC}$ is a conservative extension of $\RCA^{2}_{0}+\WKL$.  
By contrast, the fan functional implies that \emph{all} type 2-functionals are uniformly continuous, and hence contradicts $(\exists^{2})$ by \cite{kohlenbach2}*{Prop.~3.7}.

\subsection{The fan functional and continuity} 
In this section, we establish the EMT for the fan functional and derive Brouwer's continuity theorem from the latter.  
We also consider a somewhat surprising representation of the fan functional.  

\medskip

First of all, consider the following continuity principles:  
\be\label{druk}
(\forall^{\st}\varphi^{2})( \forall f^{1},g^{1}\leq_{1}1)\big[ {f}\approx_{1}{g} \di \varphi(f)=_{0}\varphi(g) \big] \tag{$\mathfrak{M}$}, 
\ee
\be\label{MUC2}\tag{$\textup{UC}^{*}$}
(\forall^{\st} \varphi^{2})(\exists^{\st}n^{0}) (\forall f^{1}, g^{1}\leq_{1}1 )[\overline{f}n=_{0}\overline{g}n\di \varphi(f)=_{0}\varphi(g)].
\ee
Here, $f^{1}\approx_{1} g^{1}$ is $(\forall^{\st}n)(f(n)=_{0}g(n))$.  In general, we say that $\varphi^{2}$ is `nonstandard continuous on Cantor space' if $(\forall f^{1},g^{1}\leq_{1}1 )\big[f \approx_{1} g \di \varphi(f)=_{0}\varphi(g) \big]$.  
\begin{thm}\label{muck}
In $\RCAO$, we have $\eqref{MUC}^{\st}\asa \eqref{MUC}\asa \eqref{druk}\asa \eqref{MUC2}$.
\end{thm}
\begin{proof}
The proof of this theorem may be found in \cite{firstHORM}*{\S4}.
By way of a sketch, to obtain $\eqref{druk}\di \eqref{MUC}^{\st}$, assume the former, define the following functional:
\be\label{dagnoor} 
\Xi(\varphi^{2},M^{0}):=(\mu y\leq M)(\forall f^{0},g^{0}\in \{0,1\}^{M})\big[(\overline{f}y=_{0}\overline{g}y) \di \varphi(f)=_{0}\varphi(g) \big], 
\ee
and note that it is $\Omega$-invariant for standard $\varphi^{2}$.  To prove this $\Omega$-invariance, it is convenient to observe that \eqref{druk} implies: 
\be\label{cruxsks}   
(\forall^{\st}\varphi^{2})(\exists^{\st}N)( \forall f^{1},g^{1}\leq_{1}1)\big[ \overline{f}N=_{0}\overline{g}N \di \varphi(f)=_{0}\varphi(g) \big] 
\ee
Using $\Omega$-CA, the standard part of $\Xi(\cdot, M)$ now yields the fan functional.  
To obtain \eqref{MUC} from \eqref{MUC}$^{\st}$, consider Remark \ref{tokkiep} and use PF-TP$_{\forall}$.  
\end{proof}
The functional $\Xi(\cdot, M)$ from \eqref{dagnoor} is called the \emph{canonical approximation} of the fan functional $\Omega(\cdot)$, and if the latter exists we have $(\forall^{\st}\varphi^{2})(\forall M\in \Omega)(\Omega(\varphi)=\Xi(\varphi, M))$.  
Arguably, this representation is much `finer' than Norman's nonstandard characterisation of the continuous functionals in \cite{jadagjan}.
Indeed, in the latter, Normann works in the semantic approach to Nonstandard Analysis and seems to freely invoke the Transfer and Standard Part principles.  Each of these three
aspects is known to yield the existence of non-computable objects, in contrast to the fact that $\RCAO$ is a conservative extension of $\RCA_{0}$.

\medskip

The representation of the \emph{non-obtainable} (standard) fan functional as the \emph{elementary computable} nonstandard object in \eqref{dagnoor} is not an isolated incident (See also Remark \ref{ohdennenboom}).  
Indeed, we now discuss another, less straightforward, approximation of the fan functional.  Indeed, the latter is defined as $\Psi(~\cdot~, \langle\rangle, \Phi)$ in \cite{bergolijf2}*{\S4}, where $\Psi$ and $\Phi$ are defined via bar recursion. 
As is typical for bar recursion, the values $\Psi(s^{0}, \dots)$ and $\Phi(s^{0},\dots)$ are defined in terms of $\Psi(t^{0}, \dots)$ and $\Phi(t^{0},\dots)$ for $|t|>|s|$, i.e.\ a potentially non-terminating recursion not expressible in $\RCAO$.    

\medskip

To guarantee that the aforementioned recursion always halt (and is expressible in $\RCAO$), 
we add an extra condition to $\Psi(s^{0}, \dots)$ and $\Phi(s^{0}, \dots)$ expressing `stop if $|s|=M$' for $M\in \Omega$.  
The canonical approximations $\Ph$ and $\Ps$ for the functionals $\Phi$ and $\Psi$ from \cite{bergolijf}*{\S4} are then defined as follows.   
Note that $\Ph$ and $\Ps$ are well-defined in $\RCAO$, as the nested recursion needed to compute them halts when the input sequence reaches length $M$.   
\bdefi[Canonical approximation]\label{fryg} Define
\[
\Ph(s,\varphi,m,M):= 
\begin{cases}
s*00\dots & |s|\geq M \\
h(s,\varphi, m,M) & \textup{otherwise}
\end{cases},
\]
where
\[
h(s,\varphi,m,M):=
\begin{cases}
\Ph(s*0, \varphi, m,M)& \varphi(s*\Ph(s*0, \varphi, m,M))\ne m\\
\Ph(s*1, \varphi, m,M)&\textup{otherwise}\\
\end{cases}.
 \]
 Define
 \[
\Ps(s,\varphi,M):= 
\begin{cases}
0 & |s|\geq M \\
g(s,\varphi,M) & \textup{otherwise}
\end{cases},
\]
where
 \[
g(s,\varphi,M):= 
\begin{cases}
0 & \begin{tabular}{l}$\textup{if } \varphi(\alpha)=\varphi(s*00\dots)$ for\\ $\alpha:=\Ph(s,\varphi,\varphi(s*00\dots),M)$\end{tabular} \\
1+\max_{i=0,1}(\Ps(s*i,\varphi,M)) & \textup{otherwise}
\end{cases}.
\]
\edefi
The following corollary to Theorem \ref{muck} is then easy to prove.  
\begin{cor}\label{fryg2}
In $\RCAO$, \eqref{druk} implies that {$\Ps(\langle\rangle,\cdot,M)$ is $\Omega$-invariant}.
\end{cor}
By the previous theorem, if the fan functional exists, it equals $\Ps(\langle\rangle,\cdot,M)$ in the standard world.  
The question if similar results exist for general bar recursive functionals, shall be explored in \cite{sambar}.  

\medskip

In light of \cite{kohlenbach2}*{Prop.\ 3.6-3.7} and the proof of the theorem, Corollary \ref{tochwelbela} below seems obvious.  Recall the usual definitions 
of real number and associated notions, introduced in Notation \ref{keepintireal}.
We consider the `positivity' property of real functions:
\be\label{POS}\tag{$\mathfrak{D}$}\textstyle
(\forall F:\R\di \R)\big[(\forall x\in [0,1])F(x)>0 \di (\exists k)(\forall x\in [0,1])F(x)>\frac{1}{k}\big].
\ee
 \begin{cor}\label{tochwelbela}
In $\RCAO$, \eqref{MUC} implies \eqref{POS}$^{\st}$ and 
\be\label{CONT}
(\forall^{\st}F:\R\di \R)(\forall x^{1},y^{1}\in [0,1])(x\approx y \di F(x)\approx F(y)). \tag{$\mathfrak{C}$}
\ee
\end{cor}
\begin{proof}
For \eqref{CONT}, define $\varphi(\alpha,k_{0})$ as that $j$ such that $\frac{j}{2^{k_{0}}}\leq [F(\sum_{i=0}^{\infty}\frac{\alpha(i)}{2^{i}})](k_{0})< \frac{j+1}{2^{k_{0}}}$, where $[z](n)=w_{n}$ for $z$ represented by the sequence $w_{n}^{1}$.  
If standard $F:\R\di \R$ does not satisfy \eqref{CONT}, there is finite $k_{0}$ such that $\varphi^{2}(\cdot,k_{0})$ is not nonstandard continuous;  Indeed, if $x_{1}\approx x_{2}$ in $[0,1]$ are such that $F(x_{1})\not\approx F(x_{2})$, then let standard $k_{0}$ be such that $|F(x_{1})-F(x_{2})|>\frac{1}{k_{0}}$ and let $\alpha_{i}\leq_{1}1$ be such that $x_{i}=\sum_{j=0}^{\infty}\frac{\alpha_{i}(j)}{2^{j}}$, i.e.\ $\alpha_{i}$ is a binary expansion of $x_{i}$.  
Note that we can choose these expansions such that $\alpha_{1}\approx_{1} \alpha_{2}$ (See \cite{polarhirst}*{p.\ 305}).  We now have $\varphi(\alpha_{1},k_{0})\ne_{0}\varphi(\alpha_{2},k_{0})$ since $|F(x_{1})-F(x_{2})|>\frac{1}{k_{0}}$.

\medskip

To establish \eqref{POS}$^{\st}$, let $F$ be as in the latter's antecedent and define $N_{0}$ as the least $n\leq M$ such that for all $i\leq M$, we have $[F(\frac{i}{M})](M)>\frac{1}{n}$.  
By assumption, $N_{0}$ is finite and we have $(\forall x\in [0,1])(F(x)>\frac{1}{2N_{0}})$ by continuity \eqref{CONT}.  
\end{proof}
The following remark on extensionality is essential for what follows.  
\begin{rem}\label{sextrem}\rm
Note that both \eqref{MUC} and \eqref{druk} immediately imply \eqref{EXT}$^{\st}$ limited to Cantor space, i.e.\ standard extensionality as follows:  
\be\label{sExt}
(\forall^{\st}\varphi^{2})(\forall^{\st}\alpha^{1},\beta^{1}\leq_{1}1)(\alpha\approx_{1}\beta \di \varphi(\alpha)=\varphi(\beta)). 
\ee
Experience bears out that this property is extremely useful, if not essential, in establishing equivalences between higher-type principles (See e.g.\ \cite{samzoo, firstHORM}).  
However, in the next section, we shall consider principles which do not (seem to) imply standard extensionality \eqref{sExt}, while the axiom \eqref{EXT}$^{\st}$ is unavailable in $\RCAO$ by \cite{brie}*{Problem 3, p.\ 1973}.  
By the following theorem, a weak version of choice suffices to remedy this absence.  
\begin{thm}\label{halleh}
In $\RCAO+\QFAC^{2,0}$, every standard functional $\varphi^{1\di 1}$ is standard extensional, i.e.\ $(\forall^{\st} f^{1},g^{1}, \varphi^{1\di 1} )(f\approx_{1}g \di \varphi(f)\approx_{1}\varphi(g))$.
\end{thm}
\begin{proof}
The axiom of extensionality for type $1\di1$-functionals implies:
\[
(\forall \varphi^{2}, f^{1}, g^{1}, k^{0})(\exists N^{0})(\overline{f}N=_{0}\overline{g}N \di \overline{\varphi(f)}k=\overline{\varphi(g)}k).
\]
Applying QF-AC$^{2,0}$, we obtain:
\be\label{hammag}
(\exists \Gamma^{3})\big[(\forall \varphi^{2}, f^{1}, g^{1}, k^{0})(\overline{f}\Gamma(\varphi, f, g, k)=_{0}\overline{g}\Gamma(\varphi, f, g, k) \di \overline{\varphi(f)}k=\overline{\varphi(g)}k)\big].
\ee
The formula in square brackets in \eqref{hammag} is internal and has no parameters but $\Gamma$, and we may assume that $\Gamma$ is standard by applying (the contraposition of) PF-TP$_{\forall}$.
For standard $\varphi^{2}, f^{1}, g^{1}$ such that $f\approx_{1}g$, we then have $\overline{f}\Gamma(\varphi, f, g, k)=_{0}\overline{g}\Gamma(\varphi, f, g, k)$ for standard $k$ as $\Gamma(\varphi, f, g, k)$ is standard.  
Hence, we have $\overline{\varphi(f)}k=\overline{\varphi(g)}k$ for all standard $k$ by \eqref{hammag}, implying $\varphi(f)\approx_{1}\varphi(g)$. 
\end{proof}
It should be noted that certain (unrelated) equivalences in \cites{firstHORM, samzoo} were proved in our base theory extended by $\QFAC^{2,0}$.  In Friedman-Simpson-style Reverse Mathematics, certain results are similarly only proved over the base theory extended with extra induction, usually $I\Sigma_{2}$ or $B\Sigma_{2}$.  Hunter notes in \cite{hunterphd}*{\S2.1.2} that any $\QFAC^{\sigma, 0}$ still results in a conservative extension of $\RCA_{0}$.        
\end{rem}
We finish this section with a remark on our choice of framework.  
\begin{rem}\label{flack}\rm
As a consequence of the above results, we observe that the fan functional $\Omega$ equals its canonical approximations $\Xi$ and $\mathbb{ps}$ from \eqref{dagnoor} and Corollary~\ref{fryg2}.  
The apparent restriction to \emph{standard input} is only a limitation of our choice of framework: Indeed, in \emph{stratified} Nonstandard Analysis, the unary predicate `$\st(x)$' is replaced by the binary predicate `$x\rel y$', to be read `$x$ is standard relative to $y$' (\cites{hrbacek3, hrbacek4, hrbacek5, aveirohrbacek, peraire}).  In this framework, we could prove the following:  
\[
(\forall \varphi^{2})(\forall M\sler \varphi)\big[\Xi(\varphi, M)=_{0}\Ps(\langle\rangle, \varphi, M)=_{0}\Omega(\varphi)\big],
\]
where $x\sler y$ is $\neg (x\rel y)$, i.e.\ $x$ is \emph{nonstandard relative} to $y$.  
In other words, in stratified Nonstandard Analysis, the canonical approximation (of the fan functional) works \emph{for any object}, not just the standard ones.    
Of course, we have chosen Nelson's framework for this paper, as this approach is more mainstream.   
\end{rem}

\subsection{Supremum functionals}\label{fafi}
In this section, we establish the EMT for the supremum functional \eqref{SUP}, defined as follows: 
\be\label{SUP}\tag{SUP}
(\exists G^{3})(\forall \varphi^{2})\big[ (\forall f^{1}\leq_{1}1 )(\varphi(f)\leq_{0} G(\varphi))\wedge (\exists g^{1}\leq_{1}1)(G(\varphi)=_{0}\varphi(g))    \big].
\ee
\be\label{druk4}
(\forall^{\st}\varphi^{2})(\exists^{\st}k^{0}_{0})\big[ (\forall f^{1}\leq_{1}1 )(\varphi(f)\leq k_{0}) \wedge (\exists^{\st} g^{1}\leq_{1}1)(k_{0}=\varphi(g))].  \tag{$\mathfrak{N}$}
\ee
Let \eqref{druk4}$^{\dagger}$ and \eqref{SUP}$^{\dagger}$ be \eqref{druk4} and \eqref{SUP} with the additional assumption that there is $g^{0}\leq_{0}1$ such that $k_{0}=\varphi(g*00\dots)$ in the second conjunct.  

\medskip

\noindent
As it turns out, \eqref{SUP} is quite similar to the principles $\tilde{F}$ and $\widehat{F}$ from \cite{kohlenbachearly}, and also to the principle $ F_{0}$ from \cite{kohlenbach8}.  
Indeed, instead of stating the existence of an upper bound which is also attained as in \eqref{SUP} and \eqref{druk4}, we could state the existence of a maximum 
as in the aforementioned axioms $\tilde{F}$, $\widehat{F}$ and ${F_{0}}$, and the equivalences from the following theorem would go through in essentially the same way.   
\begin{thm}\label{muck2}
In $\RCAO$, we have  $  (\eqref{SUP}^{\dagger})^{\st}\asa \eqref{druk4}^{\dagger}$.  In $\RCAO+\QFAC^{2,0}$: 
\be
\eqref{MUC}\asa \eqref{SUP}\asa \eqref{SUP}^{\st}\asa \eqref{druk4}\asa \eqref{SUP}^{\dagger}\asa \eqref{druk4}^{\dagger}. \label{kidi}
\ee
\end{thm}
\begin{proof}
For the equivalences in \eqref{kidi}, first  assume \eqref{MUC} and define the functional $\Gamma(\varphi):=\max_{|f^{0}|=\Omega(\varphi)\wedge f\leq_{0^{*}}1}\varphi(f*00\dots)$.  
By Theorem~\ref{muck}, \eqref{SUP}, \eqref{SUP}$^{\st}$, and \eqref{druk4}, and the daggered versions, now follow.  Next, consider \eqref{SUP} and the axiom of extensionality as follows:
\be\label{ingenu}
(\forall   f^{1},g^{1}\leq_{1}1, \varphi^{2})(\exists N^{0})\big[\overline{f}N=_{0}\overline{g}N\di \varphi(f)=_{0}\varphi(g)], 
\ee
Modulo some trivial coding, let $Y(\varphi, f,g)$ be the functional obtained from applying $\QFAC^{2,0}$ to \eqref{ingenu}.   
Define $H^{3}$ as $H(\varphi, f\oplus g)=Y(\varphi, f, g)$ and for $G$ from \eqref{SUP} consider $\Gamma(\varphi):=G(H(\varphi, \cdot))$.  By \eqref{SUP}, the previous yields:       
\be\label{belastingdruk}
(\forall  f^{1},g^{1}\leq_{1}1, \varphi^{2})(\exists N^{0}\leq \Gamma(\varphi)))\big[\overline{f}N=_{0}\overline{g}N\di \varphi(f)=_{0}\varphi(g)],   
\ee
and hence we obtain \eqref{MUC}.  Similarly, assuming \eqref{SUP}$^{\st}$, use $\QFAC^{2,0}$ to obtain \eqref{ingenu}$^{\st}$ via Theorem \ref{halleh}.  Then note that HAC$_{\textup{int}}$ implies $\QFAC^{2,0}$ relative to `st' and 
obtain \eqref{belastingdruk}$^{\st}$, and \eqref{MUC} follows by Theorem \ref{muck}.      

\medskip

Finally, to derive the remaining applications in \eqref{kidi}, assume \eqref{druk4} and consider the following two proofs:  
First of all, bring the type 1-existential quantifier in \eqref{druk4} alongside the type 0-existential quantifier, and apply HAC$_{\textup{int}}$ to obtain a standard functional $\Gamma$ such that there is $(k_{0}, g)\in \Gamma(\varphi)$ as in \eqref{druk4}.  
Note that by the second conjunct of \eqref{druk4}, we can test which is the right pair in the finite sequence $\Gamma(k_{0},g)$.  Hence, \eqref{SUP}$^{\st}$ follows and with it $\eqref{MUC}$.

\medskip

Secondly, define $\Psi(\psi,M)$ as the pair consisting of the least $k\leq M$ such that $(\forall f^{0}\leq_{0}1 )(|f|=M\wedge \psi(f*00\dots)\leq k )$, and the left-most binary $\sigma^{0}$ of least length $|\sigma|\leq M$ such that $\varphi(\sigma*00)=k$, if such exist, and $(0,\langle\rangle)$ otherwise.
To see that $\Psi(\cdot,M)$ is $\Omega$-invariant, consider standard $\psi^{2}$ and proceed as follows:  As in the previous part of the proof, obtain \eqref{ingenu}$^{\st}$ and apply $\QFAC^{2,0}$ relative to `st' to obtain the same functional $Y$.  
By \eqref{druk4}, for every standard $\varphi$ there is standard $k_{1}$ such that $Y(\varphi, f, g)\leq_{0}k_{1}$ for any binary sequences $f, g$.  Hence, we obtain 
\be\label{belastingdruk2}
(\forall^{\st}\varphi^{2})(\exists^{\st}k_{1})(\forall^{\st}  f^{1},g^{1}\leq_{1}1)\big[\overline{f}k_{1}=_{0}\overline{g}k_{1}\di \varphi(f)=_{0}\varphi(g)],   
\ee
By the continuity expressed in \eqref{belastingdruk2}, $(\exists^{\st}g^{1}\leq_{1}1)(\psi(g)=k_{0})$ implies that $(\exists^{\st}\sigma_{0}^{0}\leq_{0}1)(\psi(\sigma_{0}*00\dots)=k_{0})$, and \eqref{druk4}$^{\dagger}$ follows.  
In particular, such $\sigma_{0}$ can be taken to have length $k_{1}$, where the latter is obtained from applying \eqref{belastingdruk2} for $\psi$.  
We now observe that $\tau=\sigma_{0}*00\dots00$ with $|\tau|=M$ is one of sequences $f^{0}$ 
considered in the bounded search needed to compute $\Psi(\psi,M)$.  The assumption \eqref{druk4} implies that $\Psi(\psi,M)=\Psi(\psi,M')$ for any $M,M'\in \Omega$.  
Applying $\Omega$-CA now immediately yields \eqref{SUP}$^{\st}$ and its `dagger' version.      

\medskip

Next, the remaining applications are immediate:  To prove that \eqref{druk4}$^{\dagger}$ implies \eqref{SUP}$^{\dagger}$ relative to `st', follows from the previous part of the proof involving $\Psi$, for which obtaining \eqref{belastingdruk2} is superfluous.  
To obtain the reverse implication, note that the functional from \eqref{SUP}$^{\dagger}$ relative to `st', is uniquely defined and use PF-TP$_{\forall}$ as for \eqref{MUC}$^{\st}$ in the proof of Theorem \ref{muck} and Remark \ref{tokkiep}.   
\end{proof}
The first part of the proof reveals a subtle discrepancy between universes of standard and all objects in $\RCAO$: The former does not have extensionality but does have $\QFAC^{2,0}$, and the reverse for the latter.
Surprisingly, the latter choice axiom solves both problems.    
\begin{cor}\label{tochwelbela2}
In $\RCAO$, \eqref{MUC} or $\eqref{SUP}^{\dagger}$ implies
\begin{align}
(\forall^{\st}F:[0,1]\di \R)(\exists^{\st}y^{1}&)\big[(\forall x\in [0,1])(F(x)\lessapprox y)\notag\\
&\label{CONT2}\textstyle\wedge (\forall^{\st}k^{0})(\exists^{\st}z^{1}\in [0,1])(F(z)>_{\R}y-\frac{1}{k}) \big],  \tag{$\mathfrak{F}$}
\end{align}
while \eqref{SUP} implies the first conjunct of \eqref{CONT2}, i.e.\ that $F$ is finitely bounded. 
\end{cor}
\begin{proof}
The first implication is immediate from the theorem, Corollary~\ref{tochwelbela}, and the fact that a uniformly continuous function $F:[0,1]\di \R$ with a modulus has a supremum (See \cite{kohlenbach2}*{p.\ 293}).   
For the second implication, first of all consider the functional $\varphi(\alpha,k_{0})$ defined in terms of $F$ from the proof of Corollary~\ref{tochwelbela}.  Clearly, \eqref{SUP}$^{\st}$ implies that standard $F:[0,1]\di \R$ must be finitely bounded by considering the associated $\varphi(\alpha,k_{0})$ for $k_{0}=0$.  Secondly, to obtain \eqref{CONT2}, the second clause of \eqref{SUP}$^{\dagger}$ implies that $\varphi(\alpha,k_{0})$ attains its maximum for some $\alpha=\sigma*00\dots$ with $\sigma\leq_{0^{*}}1$ standard, i.e.\ $(\forall^{\st}k_{0})(\exists^{\st}\sigma^{0}\leq_{0}1)(\varphi(\sigma*00\dots, k_{0})=G(\varphi(\cdot, k_{0})))$, and $\QFACP$ yields $Y^{1}$ which outputs such $\sigma$.  
Finally, it is straightforward to define the supremum of $F$ using $Y$, and \eqref{CONT2} now follows from \eqref{druk4}$^{\dagger}$.  
\end{proof}
While \eqref{SUP}$^{\dagger}$ implies \eqref{CONT2} without the use of standard extensionality, it should be noted that the type-lowering modification which distinguishes \eqref{SUP}$^{\st}$ from \eqref{SUP}$^{\dagger}$, is an implicit continuity assumption.  

\medskip

Finally, in Theorem \ref{dofff} below, we prove the equivalence between \eqref{SUP} and:
\be\label{todo}
(\forall^{\st}\varphi^{2})(\forall f^{1}\leq_{1}1)(\exists^{\st}n^{0})(\varphi(f)\leq n),
\ee  
which expresses that a standard functional $\varphi^{2}$ has finite values \emph{everywhere} in Cantor space.  
Other principles have a similar equivalent formulation (See Section \ref{dilkooo}).

\subsection{Several fan theorems}
The fan functional being named after the fan theorem, it is a natural question whether there is a version of the former equivalent to the latter.  
To answer this question in the positive, we first study the \emph{quantifier-free fan theorem} QF-FAN (See e.g.\ \cite{kohlenbach3}*{p.\ 224}) and the \emph{continuous fan theorem} FAN$_{c}$ (See \cite{troelstra1}*{p.\ 80, 1.9.24}).
To avoid confusion, `fan theorem' \emph{without additional qualification} will always refer to FAN, the classical contraposition of WKL.
\subsubsection{Quantifier-free fan theorem}\label{qfffan}
The principle QF-FAN (See e.g.\ \cite{kohlenbach3}*{p.\ 224} ) is a slight generalisation of the fan theorem to quantifier-free formulas $A_{0}(f,n)$:  
\[
(\forall f^{1}\leq_{1}1)(\exists n^{0})A_{0}(f,n)\di (\exists k^{0})(\forall f^{1}\leq_{1}1)(\exists n\leq k)A_{0}(f,n).
\]
The uniform version of QF-FAN is as follows:   
\begin{align}
(\exists \Phi^{3}\in \mathfrak{L})(\forall g^{2},& H^{2})\big[(\forall \alpha^{1}\leq_{1}1)[H(\alpha,g(\alpha))=0] \notag\\
& \di (\forall \alpha^{1}\leq_{1}1)(\exists n^{0}\leq \Phi(g,H))[H(\alpha,n)=0]    \big].\label{qffan} \tag{UQF}
\end{align}
The symbolic notation `$\Phi\in \mathfrak{L}$' is short for the fact that $\Phi(g)$ is the \emph{minimal} number with the property in \eqref{qffan}.    
The nonstandard version of QF-FAN is as follows:
\begin{align}\label{nsfan}
(\forall^{\st}H^{2})\big[(\forall^{\st}\alpha^{1}\leq_{1}1)&(\exists^{\st}n^{0})(H(\alpha,n)=0) \notag\\
&\di (\exists^{\st}k^{0})(\forall \alpha^{1}\leq_{1}1)(\exists n\leq k)(H(\alpha,n)=0)   \tag{$\mathfrak{Q}$}  \big].
\end{align}
\begin{thm}\label{muck3}
In $\RCAO+\QFAC^{2,0}$, $ \eqref{MUC} \asa \eqref{qffan}^{\st}\asa \eqref{qffan}\asa \eqref{nsfan}$.  
\end{thm}
\begin{proof}
Assume $\eqref{MUC}$ and define $\Phi(g):=\max_{|\alpha^{0}|=\Omega(g) \wedge \alpha^{0}\leq_{0^{*}}1} g(\alpha*00\dots)$ to obtain \eqref{qffan}$^{\st}$ and \eqref{qffan}. 
Note in particular that $\Phi(g)$ is minimal as required.  To obtain \eqref{nsfan}, let $H$ be as in the latter's antecedent and apply $\QFACP$ to obtain $(\forall^{\st}\alpha^{1}\leq_{1}1)(H(\alpha,g(\alpha))=0)$, for some standard $g^{2}$.  
Applying \eqref{MUC} to $H(\cdot, g(\cdot))$ and $g(\cdot)$, the consequent of \eqref{nsfan} now follows.  

\medskip  
  
Assume \eqref{nsfan} and consider standard extensionality for standard $\varphi^{2}$ (Theorem~\ref{halleh}):
\be\label{sexty}
(\forall^{\st}\alpha^{1},\beta^{1}\leq 1)(\exists^{\st}N^{0})\big[\overline{\alpha}N=\overline{\beta}N\di \varphi(\alpha)=\varphi(\beta)\big], 
\ee
where the formula in square brackets may be replaced by a formula $H(\alpha, N)=0$, for standard $H^{2}$.  
By assumption, we obtain
\[
(\exists^{\st}k_{0})(\forall\alpha^{1},\beta^{1}\leq 1)(\exists N^{0}\leq k_{0})(\overline{\alpha}N=\overline{\beta}N\di \varphi(\alpha)=\varphi(\beta)),
\]
from which \eqref{druk} is immediate and we obtain \eqref{MUC} by Theorem \ref{muck}.  

\medskip

Finally, assume \eqref{qffan}$^{\st}$ and let $\varphi^{2}$ be standard.  
Define $[H(\alpha,n)=0]\equiv [\varphi(\alpha)\leq n]$ and note that $(\forall^{\st}\alpha^{1}\leq_{1}1)(H(\alpha,g(\alpha))=0)$ for $g=\varphi$.  
Applying \eqref{qffan}$^{\st}$, we observe that $\Phi(\varphi,H)$ is the supremum of $\varphi$ and \eqref{MUC} follows by Theorem \ref{muck2}.     
Similarly, \eqref{qffan} implies \eqref{SUP} and hence \eqref{MUC}.  
\end{proof} 
The original principle QF-FAN also satisfies an equivalence. 
\begin{cor}\label{tothecenter}
In $\RCAO+\QFAC^{2,0}$, $\textup{QF-FAN}^{\st}$ is equivalent to \eqref{dikjui}$^{\st}$, where: 
\be\label{dikjui}  \tag{UCS}
(\forall \varphi^{2})(\exists N)(\forall \alpha^{1},\beta^{1}\leq 1)(\overline{\alpha}N=\overline{\beta}N\di \varphi(\alpha)=\varphi(\beta)).  
\ee
The equivalence involving the internal principles holds over $\RCAo$.  
\end{cor}
\begin{proof}
For the equivalence between \eqref{dikjui}$^{\st}$ and QF-FAN$^{\st}$, apply $\QFACP$ (which follows from HAC$_{\textup{int}}$) to the antecedent of {QF-FAN}$^{\st}$ to obtain $Y^{2}$ witnessing this formula.  
Now apply \eqref{dikjui}$^{\st}$ to $Y$, yielding the consequent of QF-FAN$^{\st}$.  
For the remaining implication, consider standard extensionality as in \eqref{sexty} which follows from Theorem~\ref{halleh}, and apply QF-FAN$^{\st}$ to obtain \eqref{dikjui}$^{\st}$.  
The remaining `internal' equivalence is proved in exactly the same way.  
\end{proof}
In \cite{firstHORM}*{\S5}, it is proved that the fan theorem is equivalent to its uniform version UFAN$_{2}$ (See below) assuming $\QFAC^{2,0}$.  
In Remark \ref{bridgestone}, we sketch how such an equivalence \emph{without additional assumptions} does not work for QF-FAN.  
\begin{princ}[$\UFAN_{2}$]
There is a functional $\Phi^{3}$ such that for $T^{1}\leq_{1}1$ and $g^{2}$,
\[
(\forall \alpha^{1}\leq_{1} 1)(\overline{\alpha}g(\alpha)\not\in T)\di (\forall \alpha^{1}\leq_{1} 1)(\exists n^{0}\leq_{0} \Phi(g,T))(\overline{\alpha}n\notin T). 
\]
\end{princ}
The subscript in UFAN$_{2}$ is in place because UFAN$_{1}$, which is the former with $g$ omitted, is a different principle, namely equivalent to $(\exists^{2})$. 
Furthermore, \eqref{MUC} implies UFAN$_{2}$, and more equivalences may be found in \cite{firstHORM}*{\S5}. 
We finish this section with a sketch why QF-FAN is not equivalent to the uniform version \eqref{qffan} without invoking additional uniform principles.    
\begin{rem}\label{bridgestone}\rm  
To prove the equivalence between UFAN$_{2}^{\st}$ and FAN$^{\st}$, one notes that the latter is equivalent to FAN$^{\st}$ with the consequent weakened to 
$(\exists^{\st}k^{0})(\forall \sigma^{0}\leq_{0}1)(|\sigma|=k\di (\exists n\leq k)\sigma\not\in T)$, as trees are closed downwards.
Next, one introduces a functional $g^{2}$ witnessing the antecedent of this weak version, and one brings all quantifiers to the front.  To the resulting formula, HAC$_{\textup{int}}$ can be applied to obtain the functional from UFAN$_{2}^{\st}$ (See \cite{firstHORM}*{\S5} for details).  

\medskip

The problem with QF-FAN is that a similar weakening is not `directly' possible:  The formula $A_{0}(f, n)$ from QF-FAN need not be monotone in $n$, in contrast to the formula `$\overline{\alpha}n\not\in T$' from the fan theorem.  
Of course, assuming that the formula $A_{0}(f,n)$ is $H(f, n)=0$ for some standard $H^{2}$, we can invoke \eqref{dikjui} to prove that $H(\cdot, g(\cdot))$ has an associate $\alpha^{1}$ (See \cite{kohlenbach4}*{\S4}).  
Then, $(\forall^{\st} f^{1}\leq_{1}1)(\exists^{\st}n)H(f, n)=0$ implies $(\forall^{\st}f^{1})(\exists^{\st}n^{0})\alpha(\overline{f}n)=1$, and the latter has the right form to apply the weakening mentioned in the previous paragraph.  Thus, to obtain \eqref{qffan}$^{\st}$ in this way, we seem to require a functional which converts a (pointwise continuous) type~2-functional into an associate.  
By \cite{kohlenbach4}*{Prop.\ 4.4}, this amounts to a functional providing a modulus of pointwise continuity.  
\end{rem}
The observation made in he previous remark is one of the conceptual motivations for our study of a version of the fan functional for \emph{pointwise continuity} in Section~\ref{layola}.
\subsubsection{The continuous fan theorem}
The principle FAN$_{c}$ is a generalization of the fan theorem (\cite{kohlenbach3}*{p.\ 225} and \cite{troelstra3}*{p.\ 80, 1.9.24}) with continuity `built-in' as follows:
\[
(\forall \alpha^{1}\leq_{1}1)(\exists x^{0})A(\alpha,x)\di (\exists y^{0})(\forall \alpha^{1}\leq_{1}1)(\exists x^{0})(\forall \beta\leq_{1}1)(\overline{\alpha}y=\overline{\beta}y\di A(\beta,x).
\]
The uniform version of FAN$_{c}$ is as follows:   
\begin{align}
\big(\exists \Phi&^{2\di(0\times 2 )}\in \mathfrak{L}\big)(\forall g^{2}, H^{2})\big[(\forall \gamma^{1}\leq_{1}1)[H(\gamma,g(\gamma))=0] \di (\forall \alpha^{1},\beta^{1}\leq_{1}1) \notag\\
& [ \overline{\alpha}\Phi(g,H)(1)=\overline{\beta}\Phi(g, H)(1))\di     H(\beta,\Phi(g,H)(2)(\alpha))=0]    \big].\label{qffan2} \tag{UFC}
\end{align}
The symbolic notation `$\Phi\in \mathfrak{L}$' is short for the fact that $\Phi(g)$ provides the \emph{minimal} numbers with the property in \eqref{qffan2}.    
The `obvious' nonstandard version of FAN$_{c}$:  
\begin{align}\label{nsfan2}
(\forall^{\st}&H^{2})\big[(\forall^{\st}\alpha^{1}\leq_{1}1)(\exists^{\st}n^{0})(H(\alpha,n)=0) \notag\\
&\di (\exists^{\st}k^{0})(\forall^{\st} \alpha^{1}\leq_{1}1)(\exists^{\st} l^{0})(\forall \beta^{1}\leq_{1}1)(\overline{\alpha}k=\overline{\beta}k\di H(\beta,l)=0)   \tag{$\mathfrak{U}$}  \big].
\end{align}
Note that \eqref{nsfan2} is `self-transferring', as we can drop the `st' in $(\forall^{\st}\alpha^{1}\leq 1)$ in the antecedent.  
Finally, consider the following nonstandard version of FAN$_{c}$ which has \emph{nonstandard} continuity built-in (rather than the $\eps$-$\delta$-variety).  
\begin{align}\label{nsfan3}
(\forall^{\st}&H^{2})\big[(\forall^{\st}\alpha^{1}\leq_{1}1)(\exists^{\st}n^{0})(H(\alpha,n)=0) \notag\\
&\di (\forall \alpha^{1}\leq_{1}1)(\exists^{\st} l^{0})(\forall \beta^{1}\leq_{1}1)(\alpha \approx_{1} \beta\di H(\beta,l)=0)   \tag{$\mathfrak{W}$}  \big].
\end{align}
The three previous versions of $\FAN_{c}$ are easily seen to imply standard extensionality as in \eqref{sExt} if we consider the formula stating the totality of type 2-functionals.  
\begin{thm}\label{muck4}
In $\RCAO$, we have $ \eqref{MUC}\asa \eqref{qffan2}\asa\eqref{qffan2}^{\st}\asa \eqref{nsfan2}\asa \eqref{nsfan3}$.
\end{thm}
\begin{proof}
First of all, assume \eqref{MUC} and define $\Phi(g):=(\Omega(g), g(\overline{~\stackrel{~}{.}~}~\Omega(g)*00\dots))$.  
By the definition of the fan functional, we have for standard $\alpha^{1}, \beta^{1}\leq_{1}1, g^{2}$ that
\[
0=H(\beta,g(\beta))=H(\beta,g(\overline{\beta}\Omega(g)*00\dots))=H(\beta,g(\overline{\alpha}\Omega(g)*00\dots)),
\]
assuming the antecedent of $\eqref{qffan2}^{\st}$ and $\overline{\alpha}\Omega(g)=\overline{\beta}\Omega(g)$, i.e.\ $\eqref{qffan2}^{\st}$ follows;  The internal principle \eqref{qffan2} follows in the same way.    
To additionally obtain \eqref{nsfan2} from \eqref{MUC},
consider $h^{2}$ defined as: $h(\alpha\oplus\beta):= H(\beta,\Phi(g,H)(2)(\alpha))=0$, and consider $\Omega(h)$.     
By $\eqref{qffan2}^{\st}$, the number $k_{0}=\max(\Omega(h),\Phi(g,H)(1))$ is as in \eqref{nsfan2}.  

\medskip

Secondly, assume \eqref{nsfan2}, let $g^{2}$ be standard and consider the formula
\be\label{opli}
(\forall^{\st}\alpha^{1},\beta^{1}\leq_{1}1)(\exists^{\st}N^{0})[\overline{\alpha}N=\overline{\beta}N\di g(\alpha)=g(\beta)],    
\ee
immediate by standard extensionality.  Let $A_{0}(\alpha\oplus\beta,N)$ be the formula in square brackets in \eqref{opli}.   
By \eqref{nsfan2}, there is standard $k_{0}$ such that
\be\label{dorkuji}
(\forall^{\st}\alpha^{1},\beta^{1}\leq_{1}1)(\exists^{\st}N^{0})(\forall \xi^{1},\gamma^{1}\leq_{1}1)(\overline{\alpha\oplus\beta}k_{0}=\overline{\xi\oplus\gamma}k_{0}\di A_{0}(\xi\oplus\gamma,N)).    
\ee
Now consider $\xi^{1}_{0},\gamma^{1}_{0}\leq_{1}1$ such that $\xi_{0}\approx_{1}\gamma_{0}$ and define standard $\alpha^{1}_{0},\beta^{1}_{0}\leq_{1}1$ by 
$\alpha_{0}\oplus\beta_{0}:=\overline{\xi_{0}\oplus\gamma_{0}}k_{0}*00\dots$.  
Now apply \eqref{dorkuji} for $\alpha=\alpha_{0}$ and $\beta=\beta_{0}$ to obtain (standard) $N_{0}$ as in this formula.  By definition, we have $\overline{\alpha_{0}\oplus\beta_{0}}k_{0}=\overline{\xi_{0}\oplus\gamma_{0}}k_{0}$, implying that $A_{0}(\xi_{0}\oplus\gamma_{0},N_{0})$.  However, since $\xi_{0}\approx_{1} \gamma_{0}$, we obtain $g(\xi_{0})=g(\gamma_{0})$ from $A_{0}(\xi_{0}\oplus\gamma_{0},N_{0})$.  
Hence, $g$ is nonstandard continuous, \eqref{druk} follows, and we obtain \eqref{MUC} by Theorem \ref{muck}.  
Applying HAC$_{\textup{int}}$ also yields \eqref{MUC} `directly'.      

\medskip
  
Thirdly, assume $\eqref{qffan2}^{\st}$ or $\eqref{qffan}$ and apply the latter to $H(\alpha,n)=0$ defined as $g(\alpha)=n$.  Clearly, we have $(\forall \gamma\leq_{1}1)H(\gamma, g(\gamma))=0$ and let $\Phi$ be the functional assumed to exist.  
Then by definition, the number
\[\textstyle
\max_{|\alpha^{0}|=\Phi(g,H)(1)\wedge \alpha^{0}\leq_{0^{*}}1} \Phi(g,H)(2)(\alpha*00\dots,g) 
\]
is the supremum of $g$ and Theorem \ref{muck2} yields \eqref{MUC}.  
Finally, \eqref{nsfan2} trivially implies \eqref{nsfan3}, and to prove the remaining implication, proceed as for $ \eqref{nsfan2}\di \eqref{druk}$ in the previous part of the proof.  
\end{proof}

\section{The EMT for uniform boundedness principles}\label{UB41}
In this section, we establish the EMT for the so-called uniform boundedness principle \eqref{F} from \cite{kohlenbach3}*{Chapter~12}.  
The latter is defined as follows:
\be\label{F}
(\forall \Phi^{0\di 2}, y^{0\di 1})(\exists y_{0}\leq_{0\di 1}y)(\forall k^{0},z\leq _{1}y(k))\big[\Phi(k)(z)\leq_{0}\Phi(k)(y_{0}(k))  \big]. \tag{$F$}
\ee
This principle is called `non-standard' by Kohlenbach in \cite{kohlenbach3} as it is classically false, but we avoid this phrasing for obvious reasons. 
As to its provenance, the principle \eqref{F} finds applications in \emph{proof mining} (See e.g.\ \cites{kohlenbach4,kohlenbach5,kohlenbach6,kohlenbach7,kohlenbach8}) as a generalisation of the fan functional.  
In two words, the aim of proof mining is to extract upper bounds or similar witnessing information for existential quantifiers from (possibly non-constructive) proofs of mathematical theorems (See \cite{kohlenbach3} for an introduction).    

\medskip

The principle $\eqref{F}$ has the following nonstandard and uniform versions.
\begin{align}
\big(\exists \Theta^{((0\di 2)\times(0\di 1))\di (0\di 1)}&\big)(\forall \Phi^{0\di 2}, y^{0\di 1})(\forall k^{0})(\forall z\leq _{1}y(k))\notag\\
&\big[\Phi(k)(z)\leq_{0}\Phi(k)(\Theta(\Phi,y)(k))) \wedge \Theta(\Phi,y)\leq_{0\di 1}y  \big].  \label{UF}\tag{$\text{\emph{UF}}$}
\end{align}
\be\label{NSF}
(\forall^{\st} \Phi^{0\di 2}, y^{0\di 1})(\exists^{\st} y_{0}\leq_{0\di 1}y)(\forall^{\st} k^{0})(\forall z\leq _{1}y(k))\big[\Phi(k)(z)\leq\Phi(k)(y_{0}(k))  \big]. \tag{${F}^{*}$}
\ee
By the second conjunct in \eqref{UF}, we have $\Theta(\Phi,y)(k)\leq_{1}y(k)$ for fixed $k$, which implies that $\Phi(k)(\Theta(\Phi,y)(k))$ is a \emph{maximum} of $\Phi(k)(z)$ for $z\leq_{1}y(k)$.     
Thus, $\Theta$ is minimal in the sense of providing the least upper bound to $\Phi(k)(z)$ for $z\leq_{1}y(k)$.  

\medskip

The principle $\eqref{F}$ implies that all type $1\di 1$ objects are continuous on a bounded domain by \cite{kohlenbach3}*{Prop.\ 12.3 and Prop.\ 12.6, p.\ 226}.  
Thus, we consider the following:  
\begin{align}
\big(\exists\Psi^{((1\di 1)\times 1)\di 1})&\big(\forall \Lambda^{1\di 1}, y^{1},k^{0})(\forall z_{1},z_{2}\leq_{1}y)\notag\\
&\big[ \overline{z_{1}}\Psi(\Lambda,y)(k)=\overline{z_{1}}\Psi(\Lambda,y)(k)\di  \overline{\Lambda(z_{1})}k = \overline{\Lambda(z_{2})}k   \big].\label{UCO} \tag{UCO}
\end{align}
\begin{rem}\label{maarminnekes}\rm
Note that $\Psi(\Lambda,y)(k)$ in \eqref{UCO} can be assumed to be the \emph{least} such number for fixed $k, y, \Lambda$ (just like $\Omega(\varphi)$ from \eqref{MUC}).  
Indeed, a finite search bounded in terms of $\Psi(\Lambda,y)(k)$ and $\max_{i\leq\Psi(\Lambda,y)(k)}y(i)$ suffices to verify
whether $\Psi(\Lambda,y)(k)$ is the least number as in \eqref{UCO}.  
\end{rem}
The nonstandard versions of \eqref{UCO}$^{\st}$ are as follows:
\be
(\forall^{\st} \Lambda^{1\di 1}, y^{1})(\forall z_{1},z_{2}\leq_{1}y)\big[ {z_{1}}\approx_{1}{z_{1}}\di  {\Lambda(z_{1})}\approx_{1} {\Lambda(z_{2})}   \big].\label{UC} \tag{$\mathfrak{G}$}
\ee
\be
(\forall^{\st} \Lambda^{1\di 1}, y^{1})(\exists^{\st}\xi^{1})(\forall^{\st}  k)(\forall z,w\leq_{1}y)\big[ \overline{z}\xi(k)=\overline{w}\xi(k)\di  \overline{\Lambda(z)}k = \overline{\Lambda(w)}k   \big].\label{UD} \tag{$\mathfrak{H}$}
\ee 
Clearly, the three previous continuity statements imply standard extensionality for standard type $1\di1$-functionals as follows:
\be\label{Extmore}  
(\forall^{\st} \Lambda^{1\di 1},f^{1},g^{1})(f\approx_{1}g\di \Lambda(f)\approx_{1} \Lambda(g)),
\ee
which also follows from Theorem \ref{halleh} above.  

\medskip

By \cite{kohlenbach3}*{Prop.\ 12.7}, the seemingly weaker axiom $F^{-}$ can be derived from $F$ given QF-AC$^{1,0}$.  Hence, we could consider uniform and nonstandard versions of $F^{-}$, which would be equivalent to \eqref{UF} too.   
\begin{thm}\label{muck8}
In $\RCAO+\QFAC^{2,0}$, we have 
\[
 \eqref{UF}\asa\eqref{UF}^{\st}\asa  \eqref{NSF}\asa \eqref{UCO}^{\st}\asa \eqref{UCO}\asa \eqref{UC}\asa \eqref{UD}.
\]
The extra axiom of choice is only necessary for the third forward implication.  
\end{thm}
\begin{proof}
The equivalences $\eqref{UC}\asa \eqref{UD}\asa\eqref{UCO}^{\st}\asa \eqref{UCO}$ are proved in the same way as $\eqref{MUC}^{\st}\asa \eqref{MUC}\asa \eqref{druk}\asa \eqref{MUC2}$ in Theorem \ref{muck} and Remark \ref{tokkiep}, hence we shall be brief.  
First of all, as noted in Remark \ref{maarminnekes}, we may assume $\Psi(\Lambda, y)(k)$ as in \eqref{UCO} is the least number as in the latter.  Using PF-TP$_{\forall}$, we easily obtain $\eqref{UCO} \asa \eqref{UCO}^{\st}$ and that the former implies \eqref{UC}.    
Furthermore, define the functional $\Theta(\Lambda,y,M)$ as follows: For any $k$, $\Theta(\Lambda,y,M)(k)$ is equal to:
\[
(\mu N\leq M)(\forall z^{0},w^{0}\leq_{0}\overline{y}M)\big[(|z|,|w|=M \wedge \overline{z}N=\overline{w}N)\di \overline{\Lambda(z)}k=\overline{\Lambda(w)}k\big].
\]
Now this functional is $\Omega$-invariant given \eqref{UC}, as the latter implies:
\be\label{cruxsks2}
(\forall^{\st} \Lambda^{1\di 1}, y^{1},k^{0})(\exists^{\st}N)(\forall z_{1},z_{2}\leq_{1}y)\big[ \overline{z_{1}}N=\overline{z_{1}}N\di  \overline{\Lambda(z_{1})}k= \overline{\Lambda(z_{2})}k   \big], \tag{$\mathfrak{K}$}
\ee
in the same way as \eqref{druk} implies \eqref{cruxsks}.  By $\Omega$-CA, \eqref{UCO}$^{\st}$ now follows.  Finally, \eqref{UD} clearly implies \eqref{UC}, while \eqref{cruxsks2} implies the latter by applying HAC$_{\textup{int}}$.         

\medskip
  
For the remaining equivalences, we first prove $\eqref{UCO}\di \eqref{NSF}$.  To this end, fix standard $y^{0\di 1}, \Phi^{0\di 2}, k^{0}$, define $\Lambda^{1\di 1}$ as $\Lambda(z):= (\Phi(k)(z),\Phi(k)(z),\dots)$ for $z^{1}$, and define $y_{1}:=y(k)$.  
Now let the (standard by PF-TP$_{\forall}$) functional $\Psi$ be as in $\eqref{UCO}$, i.e.\ for standard $\xi^{1}:=\Psi(\Lambda,y_{1})$ we have 
\[
(\forall^{\st}  l)(\forall z,w\leq_{1}y_{1})\big[ \overline{z}\xi(l)=\overline{w}\xi(l)\di  \overline{\Lambda(z)}l = \overline{\Lambda(w)}l  \big],
 \]
 implying by definition that
\be\label{onderbuik}
(\forall^{\st}  l)(\forall z,w\leq_{1}y(k))\big[ \overline{z}\xi(l)=\overline{w}\xi(l)\di  \Phi(k)(z)= \Phi(k)(w) \big].
\ee
Now we obtain the required $y_{0}^{0\di 1}$ by defining $y_{0}(k)$ as $z_{0}*00\dots$ where $|z_{0}|=\xi(1)\wedge z_{0}\leq_{0}\overline{y(k)}\xi(1)$ and $\Phi(k)(z_{0}*00)=\max_{|w|=\xi(1)\wedge w\leq_{0}\overline{y(k)}\xi(1)}\Phi(k)(w*00\dots)$, and \eqref{NSF} follows.   
Furthermore, the implication $\eqref{UCO}\di \eqref{UF}^{\st}$ follows by putting $\Theta(\Phi,y)(k):=y_{0}(k)$ as defined above.  

\medskip 
 
Next, to prove that $\eqref{UF}^{\st}\di \eqref{NSF}$, proceed as in the first part of the proof: Obtain $\eqref{UF} \asa \eqref{UF}^{\st}$ using PF-TP$_{\forall}$, and the former immediately implies \eqref{NSF}.  
Finally, assume \eqref{NSF} and consider for standard $\Lambda^{1\di 1}$ and $y^{1}$,
\be\label{kokje}
(\forall^{\st}z_{1},z_{2}\leq_{1}y)(\forall^{\st}k^{0})(\exists^{\st}N^{0})\big[\overline{z_{1}}N=\overline{z_{2}}N\di \overline{\Lambda(z_{1})}k = \overline{\Lambda(z_{2})}(k)    \big],
\ee
which follows easily from the standard extensionality of $\Lambda$.  Now let $Y^{(1\times 1\times 0)\di 0}$ be obtained from applying $\QFACP$ to a properly coded 
version of \eqref{kokje} and define $\Phi^{0\di 2}$ as $\Phi(k)(z\oplus w):=Y(z,w,k)$, and $w^{0\di 1}$ as $(y,y,\dots)$.  By  \eqref{NSF}, 
there is standard $y_{0}^{0\di 1}\leq_{0\di 1}w$ such that $(\forall^{\st} k^{0})(\forall z\leq _{1}w(k))\big[\Phi(k)(z)\leq_{0}\Phi(k)(y_{0}(k))  \big]$.  The latter implies by definition that 
$(\forall^{\st} k^{0})(\forall z_{1},z_{2}\leq _{1}y)\big[Y(z_{1},z_{2},k))\leq_{0} N_{0}(k)  \big]$, for $N_{0}(k):=\Phi(k)(y_{0}(k))$, which does not involve $z_{1},z_{2}$.  
Hence, \eqref{kokje} implies 
\[
(\forall z_{1},z_{2}\leq_{1}y)(\forall^{\st}k^{0})\big[\overline{z_{1}}N_{0}(k)=\overline{z_{2}}N_{0}(k)\di \overline{\Lambda(z_{1})}k = \overline{\Lambda(z_{2})}(k)    \big],
\]
immediately implying \eqref{UC} as $N_{0}(\cdot)$ is standard.  We could also apply HAC$_{\textup{int}}$ to obtain \eqref{UCO}$^{\st}$ directly.    
\end{proof}
\begin{rem}\label{tuni}\rm
Similar to \eqref{SUP}$^{\dagger}$ and \eqref{druk4}$^{\dagger}$ from Section \ref{fafi}, we could obtain `daggered' versions of \eqref{UF}$^{\st}$ and \eqref{NSF} by lowering the type of 
the objects claimed to exist by the latter;  This is possible in light of the definition of $\Theta$ below \eqref{onderbuik}.  
These versions would be equivalent without the use of standard extensionality.  Furthermore, using the functional $\Psi$ from \eqref{UCO}, it is easy to obtain the supremum of $\Lambda^{1\di 1}$ as in this principle.
Thus, we could consider a version of \eqref{SUP}$^{\st}$ for $\Lambda^{1\di 1}$ involving `$\leq_{1}$', and obtain results similar to Theorem~\ref{muck2}.     
\end{rem}
An alternative uniform boundedness principle used extensively in proof mining is $\Sigma^{0}_{1}$-UB (\cite{kohlenbach2}*{Def.\ 12.1}).  
The uniform version of $\Sigma_{1}^{0}$-UB is as follows:  
\begin{align}
(\exists \Psi& \in \mathfrak{L})(\forall y^{0\di 1}, H,g)\big[(\forall k^{0})(\forall x^{1}\leq_{1}y(k))[H(x,y,g(x,y,k),k)=0] \notag\\
 & \di(\forall k^{0})(\forall x^{1}\leq_{1}y(k))(\exists z^{0}\leq_{0}\Psi(y,g)(k))[H(x,y,z,k)=0]    \big].\label{USB} \tag{USB}
\end{align}
Again `$\Psi\in \mathfrak{L}$' means that $\Psi(y,g)$ is the least number with the property as in \eqref{USB}.  
The nonstandard version is as follows:
\begin{align}
(\forall^{\st}& y^{0\di 1}, H)\big[(\forall^{\st} k^{0})(\forall^{\st} x^{1}\leq_{1}y(k))(\exists^{\st} z^{0})[H(x,y,z,k)=0] \notag\\
 & \di(\exists^{\st}\xi^{1})(\forall^{\st} k^{0})(\forall x^{1}\leq_{1}y(k))(\exists z^{0}\leq_{0}\xi(k))[H(x,y,z,k)=0]    \big].\label{NSB} \tag{$\mathfrak{S}$}
\end{align}

\begin{thm}
In $\RCAO+\QFAC^{2,0}$, we have $\eqref{UF}\asa\eqref{NSB}\asa \eqref{USB}$.
\end{thm}
\begin{proof}
To establish $\eqref{NSB}\di \eqref{UD}$, let $\Lambda^{1\di 1},y^{1}$ be as in the latter and derive from standard extensionality that 
\be\label{qoded}
(\forall^{\st} z_{1},z_{2}\leq_{1}y)(\forall^{\st} k^{0})(\exists^{\st}n^{0})\big[\overline{z_{1}}n=\overline{z_{2}}n\di \overline{\Lambda(z_{1})}k=\overline{\Lambda(z_{2})}k\big],
\ee
and apply \eqref{NSB} to (a properly coded version of) \eqref{qoded} to obtain \eqref{UD}.  

\medskip

To establish $\eqref{UD}\di \eqref{NSB}$, fix $H,y^{0\di 1}$ as in the latter and suppose 
$(\forall^{\st} k^{0})(\forall^{\st} x^{1}\leq_{1}y(k))(\exists^{\st} z^{0})[H(x,y,z,k)=0]$.  Apply $\QFACP$ to obtain $Y$ such that we have $(\forall^{\st} k^{0})(\forall^{\st} x^{1}\leq_{1}y(k))[H(x,y,Y(x,k),k)=0]$.
Fix standard $k_{0}$ and apply \eqref{UD} to $\Lambda^{1\di 1}:=(Y(x,k_{0}),Y(x,k_{0}),\dots)$ to obtain
\[
(\exists^{\st}\phi^{1})(\forall^{\st}  l)(\forall z,w\leq_{1}y(k_{0}))\big[ \overline{z}\phi(l)=\overline{w}\phi(l)\di  \overline{Y(z,k_{0})}l = \overline{Y(w,k_{0})}l   \big], 
\]
implying that (take $N_{0}=\phi(1)$)
\[
(\forall^{\st}k)(\exists^{\st}N_{0})(\forall z,w\leq_{1}y(k))\big[ \overline{z}N=\overline{w}N\di  {Y(z,k)} = {Y(w,k)}  \big]. 
\]
Now use HAC$_{\textup{int}}$ to obtain a standard function $g^{2}$ witnessing the existential 
quantifier in the previous formula, i.e.\  
\[
(\forall^{\st}k)(\forall z,w\leq_{1}y(k))\big[ \overline{z}g(k)=\overline{w}g(k)\di  {Y(z,k)} = {Y(w,k)}.
\]    
Finally, define $\xi(k)$ needed to establish \eqref{NSB} as the maximum of all $Y(z*00\dots,k)$ for $|z|=g(k) \wedge (\forall i\leq |z|)((z(i)\leq y(k)(i)))$.  

\medskip

To prove the equivalence with \eqref{USB}, one proceeds as for \eqref{qffan} in the proof of Theorem \ref{muck3}.  
In particular, one can prove versions of \eqref{MUC} and \eqref{SUP} for `$\leq_{1}y(k)$' instead of `$\leq_{1}1$' from \eqref{USB}.  
Similarly, \eqref{UCO} implies a generalization of \eqref{MUC} (as mentioned just now) and from this \eqref{USB} follows in the same way as \eqref{MUC} implies \eqref{qffan}.     
\end{proof}
One could consider the generalisations of $(F)$ and $\Sigma_{1}^{0}$-UB to higher types (See \cite{kohlenbach3}*{Def.\ 12.11}), and obtain similar results.  

\section{The EMT for weak and pointwise continuity principles}\label{SB}
In this section, we establish the EMT for principles which are weaker than the fan functional.  
Besides proving `more of the same' concerning EMT, this investigation will also yield Theorem \ref{allisone} in which we obtain a large number of equivalences.  
Furthermore, this study will also give rise to Remark \ref{loofer}, in which we show that a \emph{higher-order} principle is implicit in second-order RM, due to the definition of continuity used in the latter.  
Finally, a conceptual motivation for the results in the section was also provided in Remark~\ref{bridgestone}.  

\medskip

In Section \ref{main}, we considered the Reverse Mathematics of the fan functional.  As the latter deals with uniform continuity, a natural question is what happens if we limit ourselves to \emph{pointwise continuity}, 
i.e.\ a modulus-of-continuity-functional as in \eqref{POC}.  Another natural avenue of research is to consider the \emph{weak continuity for numbers} principle \eqref{WCN} as the latter is `that other' non-arithmetical principle of intuitionistic mathematics besides bar induction (\cite{atjeendaaltje}*{p.\ 329}).    
\begin{align}\label{WCN}\tag{WC-N}
(\forall \alpha^{1})(\exists n^{0})&A(\alpha, n)\di 
 (\forall \alpha^{1})(\exists m^{0},x^{0})(\forall \beta^{1})(\overline{\alpha}m=\overline{\beta}m\di A(\beta,x)).\notag
\end{align}
Before we can study these principles, we note that the existence of the fan functional \eqref{MUC} and its equivalent formulations all (classically) imply $\WKL$, which is a kind of compactness principle.  
In the absence of the latter, we shall need the following \emph{weak compactness property}, defined as:  For all 
internal quantifier-free $A_{0}$, we have
\be\label{WCP}\tag{WCP}
(\forall f^{1})(\exists^{\st}n^{0})A_{0}(f,n)\di (\exists^{\st} k^{0})(\forall f^{1})(\exists n^{0}\leq k)A_{0}(f,n).
\ee
Constructively, \eqref{WCP} follows from the so-called non-classical {realization principle} NCR (\cite{brie}*{p.\ 1971}), while classically \eqref{WCP} follows by contraposition  from the idealisation axiom~I of $\RCAO$

\subsection{Local pointwise continuity}\label{LOCO}
In this section, we study the following variants of the fan functional dealing with pointwise continuity.     
\be\label{POC}\tag{MPC}
(\exists \Delta^{3})(\forall \varphi^{2}) (\forall f^{1}, g^{1}\leq_{1}1 )[\overline{f}\Delta(\varphi,f)=_{0}\overline{g}\Delta(\varphi,f)\di \varphi(f)=_{0}\varphi(g)].
\ee
\be\label{POCS}\tag{PC$^{*}$}
(\forall^{\st} \varphi^{2},f^{1}\leq_{1}1)(\exists^{\st}k^{0})(\forall g^{1}\leq_{1}1 )[\overline{f}k=_{0}\overline{g}k\di \varphi(f)=_{0}\varphi(g)].
\ee
\be\label{POCSS}\tag{$\mathfrak{PC}$}
(\forall^{\st} \varphi^{2},f^{1}\leq_{1}1)(\forall g^{1}\leq_{1}1 )[f\approx_{1}g\di \varphi(f)=_{0}\varphi(g)].
\ee
Note that \eqref{POC} is related to C-N as the latter expresses the existence of a modulus-of-continuity functional (See \cite{troelstra1}*{p.\ 77}).  
Furthermore, since \eqref{POC} is false in ECF, a model of $\RCAo$, the former principle is not provable in the latter system (See \cite{troelstra1}*{2.6.7, p.\ 142} and \cite{kohlenbach2}*{Proof of Prop.\ 3.1}).

\begin{thm}\label{trikkeagain}
In $\RCAO$, we have $\eqref{POC}\asa \eqref{POCS} \asa\eqref{POCSS}\asa \eqref{POC}^{\st}$.    
\end{thm}
\begin{proof}
First of all, the implication $\eqref{POC}\di \eqref{POCS}$ follows by applying PF-TP$_{\forall}$ to the former principle.      
Now assume \eqref{POCS} and apply HAC$_{\textup{int}}$ to obtain standard $\Phi$ such that $(\exists k\in \Phi(\varphi, f))$ as in \eqref{POCS}.  Define $\Delta(\varphi, f)$ 
as the maximum of $\Phi(\varphi, f)(i)$ for $i<|\Phi(\varphi, f)|$ and note that we obtain \eqref{POC}$^{\st}$.  
Hence, the antecedent of \eqref{dorg} in Remark~\ref{tokkier} holds and $\Lambda_{0}^{3}$ satisfies the consequent of the former formula.  However, by the definition of associate, $\Lambda_{0}(\varphi, \cdot)$ is also a modulus of continuity of standard $\varphi^{2}$, i.e.\ we have $K^{\st}(\Lambda_{0}(\cdot))$.  Since the latter universal formula does not involve any parameters, PF-TP$_{\forall}$ yields $K(\Lambda_{0}(\cdot))$, which is \eqref{POC}.  

\medskip

Secondly, while \eqref{POCS} trivially implies \eqref{POCSS}, the reverse implication requires \eqref{WCP}.  Thus, assume \eqref{POCSS}, and note that 
$(\forall g^{1}\leq_{1}1 )[f\approx_{1}g\di \varphi(f)=_{0}\varphi(g)$ for fixed standard $\varphi $ and $f$, implies $(\forall g^{1}\leq_{1}1 )(\exists^{\st}N)[\overline{f}N=\overline{g}N\di \varphi(f)=_{0}\varphi(g)$ by definition, and also $(\exists^{\st}k)(\forall g^{1}\leq_{1}1 )(\exists N\leq k)[\overline{f}N=\overline{g}N\di \varphi(f)=_{0}\varphi(g)$ by \eqref{WCP}, and \eqref{POCS} follows.  
\end{proof}
\begin{cor}\label{tochwelbel3a}
In $\RCAO$, \eqref{POC} implies
\be\label{CONTK}
(\forall^{\st}F:\R\di \R)(\forall^{\st} x^{1})(\forall y^{1})(x\approx y \di F(x)\approx F(y)).  \tag{$\mathfrak{E}$}
\ee
\end{cor}
\begin{proof}
Immediate from the proof of Corollary \ref{tochwelbela}.  
\end{proof}
In light of the proofs in this section, it seems that the idealization axiom in the form \eqref{WCP} is essential to obtaining the associated theorems.  
The power of this axiom is that it can `push standard type 0-existential quantifiers through universal quantifiers', intuitively speaking.  However, this means we can also 
use \eqref{WCP} to obtain principles like $T^{*}$ from suitable nonstandard principles.  We now consider one example, and more are given in Section \ref{layola}.  
\begin{cor}\label{kloppip}
In  $\RCAO$, \eqref{MUC} is equivalent to
\be\label{donkio}
(\forall^{\st}\varphi^{2})(\forall f^{1}\leq_{1}1)(\exists^{\st}k^{0})(\forall^{\st}g^{1}\leq_{1}1)(\overline{f}k=\overline{g}k\di \varphi(f)=\varphi(g)), \tag{$\mathfrak{R}$}
\ee
i.e.\ standard pointwise continuity at \pmb{every} point of Cantor space.
\end{cor}
\begin{proof}
Note that \eqref{donkio} implies standard extensionality on Cantor space.  
We only need to prove \eqref{MUC} from \eqref{donkio}.  Clearly, the latter implies that
\be\label{proffff}
(\forall^{\st}\varphi^{2})(\forall f^{1}\leq_{1}1)(\exists^{\st}k^{0})(\varphi(f)=\varphi(\overline{f}k*00\dots)),
\ee
since $\overline{h}k*00\dots$ is standard for standard $k$ and any $h^{1}\leq_{1}1$.  Apply \eqref{WCP} to \eqref{proffff} to obtain that
\[
(\forall^{\st}\varphi^{2})(\exists^{\st}n^{0})(\forall f^{1}\leq_{1}1)(\exists k \leq n)(\varphi(f)=\varphi(\overline{f}k*00\dots)),
\] 
which in turn yields (for the same $n^{0}$) that
\[
(\forall^{\st}\varphi^{2})(\exists^{\st}n^{0})(\forall f^{1}\leq_{1}1)(\varphi(f)\leq \max_{|\sigma|\leq n, \sigma{\leq_{0^{*}}1}} \varphi(\sigma*00\dots)).
\] 
The previous formula clearly implies \eqref{druk4}$^{\dagger}$, and using the proof of Theorem \ref{muck2}, we obtain \eqref{MUC}, finishing this proof.
\end{proof}
Note that we could define an equivalent uniform version of \eqref{donkio}, similar to \eqref{POC}$^{\st}$.  
Furthermore, the following principle is a version of \eqref{UD} similar to \eqref{donkio}.  
\[
(\forall^{\st} \Lambda^{1\di 1}, y^{1}, k^{0})(\forall z\leq_{1}y)(\exists^{\st}N)(\forall^{\st}w\leq_{1}y )\big[ \overline{z}N=\overline{w}N\di  \overline{\Lambda(z)}k = \overline{\Lambda(w)}k   \big].
\]
As in the corollary, the previous formula is equivalent to \eqref{UD}, and proving this seems to require a version of \eqref{SUP}$^{\st}$ for $\Lambda^{1\di 1}$, as described in~Remark \ref{tuni}.   
More examples are considered in Section \ref{layola}.    

\medskip

Finally, let $\eqref{POC}_{0}$ and \eqref{POCS}$_{0}$ be the `classical' versions of \eqref{POC} and \eqref{POCS}, i.e.\ the latter principles with $(\forall \varphi^{2})$ replaced by $(\forall \varphi^{2}\in C)$, where `$\varphi^{2}\in C$' is short for pointwise continuity on Cantor space, i.e. 
\[
(\forall f^{1}\leq_{1}1)(\exists N)(\forall g^{1}\leq_{1}1)(\overline{f}N=_{0}\overline{g}N\di \varphi(f)=_{0}\varphi(g)).
\]  
As in Theorem \ref{trikkeagain}, one proves that $\eqref{POC}_{0}\asa \eqref{POCS}_{0}$.  We now argue that the latter principle, and hence apparently the former, is actually implicit in second-order RM due to the RM-definition of continuity. 
This was first observed in \cite{samimplicit}.
\begin{rem}[Continuity in Reverse Mathematics]\label{loofer}\rm
Friedman-Simpson style Reverse Mathematics takes place in second-order arithmetic, i.e.\ only type $0$ and $1$ objects are available.  
As a result, one cannot define real-valued functions `directly' as the latter objects have type $1\di 1$.  
For this reason, a real-valued continuous function is represented in Reverse Mathematics by a \emph{code} as in \cite{simpson2}*{II.6.1}, a notion closely related 
the definition of an \emph{associate} as in \cite{kohlenbach4}*{Def.\ 4.3}. 

\medskip

By \cite{kohlenbach4}*{Prop.~4.4}, the Reverse Mathematics definition of continuity (for higher type objects) corresponds to pointwise continuity \emph{with a continuous modulus of continuity}, i.e.\ the definition of continuity used in Reverse Mathematics involves a slight constructive enrichment compared to the `epsilon-delta' definition.  However, by \cite{kohlenbach4}*{Prop.~4.10}, this enrichment does not affect the Reverse Mathematics of $\WKL_{0}$.  
We now show that codes also gives rise to a \emph{nonstandard} enrichment.  

\medskip 

Since the Reverse Mathematics definition of continuity implicitly involves a modulus, we shall make the latter explicit.  
Hence, we represent a continuous function $\phi$ on Cantor space via a pair of codes $(\alpha^{1}, \beta^{1})$, where $\alpha$ codes $\phi$ and $\beta$ codes a continuous modulus of pointwise continuity $\omega_{\phi}$ of $\phi$.      
In more technical detail, $\alpha$ and $\beta$ satisfy $(\forall \gamma^{1}\leq_{1}1)(\exists N^{0})\alpha(\overline{\gamma}N)>0$ and $(\forall \gamma^{1}\leq_{1}1)(\exists N^{0})\beta(\overline{\gamma}N)>0$;  
The values of $\omega_{\phi}$ and $\phi$ at $\gamma^{1}\leq_{1}1$, denoted $\omega_{\phi}(\gamma)$ and $\phi(\gamma)$, are $\beta(\overline{\gamma}k)-1$ and $\alpha(\overline{\gamma}k)-1$ for any $k^{0}$ such that the latter numbers are at least zero.  
Now the following formula makes sense and expresses that $\omega_{\phi}$ is the modulus of continuity of $\phi$:
\be\label{krif2}
(\forall \zeta^{1}, \gamma^{1}\leq_{1}1)(\overline{\zeta}\omega_{\phi}(\zeta)=\overline{\gamma}\omega_{\phi}(\zeta)\di \phi(\zeta)=\phi(\gamma)).
\ee  
However, to represent a \emph{standard} continuous function $\phi$ on Cantor space, we should require that $\phi$ and $\omega_{\phi}$ satisfy the basic axioms $\mathcal{T}_{\st}$ 
(See \cite{bennosam}*{\S2}) of $\RCAO$, in particular that $\phi(\gamma)$ and $\omega_{\phi}(\gamma)$ are standard for standard $\gamma^{1}\leq_{1}1$.  
To accomplish this, we require that $\alpha$ and $\beta$ are standard and that these codes additionally satisfy:
\begin{align}\label{kruks}
(\forall^{\st} \gamma^{1}\leq_{1}1)(\exists &N^{0})(\exists^{\st}K^{0})[K\geq \alpha(\overline{\gamma}N)>0] \\
&\wedge(\forall^{\st} \gamma^{1}\leq_{1}1)(\exists N^{0})(\exists^{\st}K^{0})[K\geq \beta(\overline{\gamma}N)>0].\notag
\end{align}
Obviously, there are other ways of guaranteeing that $\phi$ and $\omega_{\phi}$ map standard binary sequences to standard numbers.
Whichever way we guarantee that $\omega_{\phi}$ and $\phi$ are standard for standard input, \eqref{krif2} yields that 
\be\label{krif3}
(\forall^{\st} \zeta^{1}\leq_{1}1)(\exists^{\st}N)(\forall  \gamma^{1}\leq_{1}1)(\overline{\zeta}N=\overline{\gamma}N\di \phi(\zeta)=\phi(\gamma)),
\ee
since $\omega_{\phi}(\zeta)$ is assumed to be standard for standard binary $\zeta^{1}$.  Note that \eqref{krif3} implies that $\phi$ is also \emph{nonstandard} pointwise continuous, i.e.\ 
\[
(\forall^{\st} \zeta^{1}\leq_{1}1)(\forall  \gamma^{1}\leq_{1}1)({\zeta}\approx_{1}{\gamma}\di \phi(\zeta)=\phi(\gamma)),
\]
which is the `nonstandard enrichment' we hinted at previously.  
In conclusion, for standard and continuous $\phi$ on Cantor space, we have \eqref{krif3}, which is exactly \eqref{POCS}$_{0}$ for coded functions $\phi$ on Cantor space.    
Hence, we observe that the uniform principle \eqref{POC}$_{0}$ is implicit in second-order RM, due to the special nature of the RM-definition of continuity.  
\end{rem}
\subsection{Weak and global continuity}
Consider the following nonstandard and uniform versions of the weak continuity principle \eqref{WCN}.  
\begin{align}\label{UWC}\tag{UWC}
(\exists \Psi^{(1\times 2)\di (0\times 0)}&)(\forall H, \psi^{2})\big[(\forall f^{1})(H(f,\psi(f))=0)\di \\
&(\forall f^{1}, g^{1})(\overline{f}\Psi(f,\psi)(1)=\overline{g}\Psi(f,\psi)(1)\di H(g,\Psi(f,\psi)(2))=0)\big].\notag
\end{align}
Note that $\Psi(f,\psi)(1)$ can be assumed to be the least\footnote{In other words, there is a binary sequence $h^{0}$ such that $|h|=\Psi(f,\psi)_{1}$ and 
$\overline{f}\Psi(f,\psi)_{1}-1=\overline{h}\Psi(f,\psi)_{1}-1\wedge H(h*00\dots,\Psi(f,\psi)_{2})\ne0$} number as in \eqref{UWC}.  
\begin{align}\label{WCS}\tag{WC$^{*}$}
(\forall^{\st}H)\big[(\forall^{\st}f^{1}&)(\exists^{\st}n^{0})(H(f,n)=0)\di \\
&(\forall^{\st}f^{1})(\exists^{\st}m^{0},x^{0})(\forall g^{1})(\overline{f}m=\overline{g}m\di H(g,x)=0)\big].\notag
\end{align}
\begin{align}\label{WCSS}\tag{$\mathfrak{WC}$}
(\forall^{\st}H)\big[(\forall^{\st}f^{1}&)(\exists^{\st}n^{0})(H(f,n)=0)\di \\
&(\forall^{\st}f^{1} )(\exists^{\st}x^{0})(\forall g^{1})(f\approx_{1}g\di H(g,x)=0)\big].\notag
\end{align}
It is not difficult to show that the previous three principles \emph{limited to Cantor space} are equivalent to \eqref{POC}.  
In \cite{troelstra1}*{\S1.9.19, p.\ 77}, Troelstra also notes that \eqref{WCN} gives rise to certain continuity conditions for type $2$-functionals.  
Thus, we consider the following continuity principles:  
\be\label{CONT1337}\tag{CONT}
(\exists \Psi^{3})(\forall \varphi^{2}, f^{1}, g^{1})\big[\overline{f}\Psi(\varphi, f)=_{0}\overline{g}\Psi(\varphi, g)\di \varphi(f)=\varphi(g) \big].
\ee
\be\label{CONT1338}\tag{$\mathfrak{CO}$}
(\forall^{\st} \varphi^{2}, f^{1})(\forall g^{1})\big[f\approx_{1}g \di \varphi(f)=\varphi(g) \big].
\ee
\be\label{CONT1339}\tag{CO$^{*}$}
(\forall^{\st} \varphi^{2}, f^{1})(\exists^{\st}N^{0})(\forall g^{1})\big[\overline{f}N=_{0}\overline{g}N\di \varphi(f)=\varphi(g) \big].
\ee
Note that \eqref{CONT1337} is related to C-N as the latter expresses the existence of a modulus-of-continuity functional according to Troelstra (See \cite{troelstra1}*{p.\ 77}).  
In particular, a modulus of (pointwise) continuity can be uniformly converted into an associate (as in $K_{0}$ in \cite{troelstra1}*{p.\ 77}) by the proof of \cite{kohlenbach4}*{Prop.\ 4.4}.  
We assume a version of \eqref{dorg} corresponding to \eqref{CONT1337} has been added to $\RCAO$.  
\begin{thm}
In $\RCAO$, we have 
\[
\eqref{UWC}\asa \eqref{UWC}^{\st}\asa \eqref{WCS}\asa \eqref{WCSS}\asa \eqref{CONT1337}\asa \eqref{CONT1337}^{\st}\asa \eqref{CONT1338}\asa \eqref{CONT1339}.  
\]
\end{thm}
\begin{proof}
First of all, the equivalence between \eqref{WCS} and \eqref{WCSS} (and \eqref{CONT1338} and \eqref{CONT1339}) is proved as for \eqref{POCS} and \eqref{POCSS} in the previous proof.   
In general, the first three and the last three equivalences in the theorem are proved similarly to the proofs of the previous theorems.  We shall only establish the remaining equivalence.  
To prove that $\eqref{WCS}\di \eqref{CONT1339}$, apply the former to $(\forall^{\st}f^{1})(\exists^{\st}n)(\varphi(f)=n)$ for standard $\varphi^{2}$.  
The reverse implication follows by applying the \eqref{CONT1339} to $H(\cdot,g(\cdot))$.  
\end{proof}
In light of the above results, nonstandard continuity may be qualified as `standard continuity with a modulus'.    
\section{Reverse Mathematics of Brouwer's continuity theorem}\label{strongEMT}
In this section, we use the above results to obtain the Reverse Mathematics classification of Brouwer's continuity theorem, assuming (weakenings of) \eqref{POC}.   
In light of \cite{kohlenbach4}*{Prop.\ 4.8-4.9}, this assumption seems to be rather weak.  As argued in Remark \ref{bridgestone}, the assumption \eqref{POC} seems essential to connect uniform and non-uniform intuitionistic principles.  We also obtain some natural splitting results for the fan functional in the next section.      
\subsection{The fan theorems}\label{fans}
In this section, we prove preliminary results involving the fan theorem as a step towards classifying Brouwer's continuity theorem.  
Certain results are interesting in their own right, as we obtain a `splitting' of the fan functional into various pairs of equally natural principles.  
As discussed in \cites{montahue, schirfeld}, such splitting results are sought after in Reverse Mathematics.  

\medskip

First of all, in \cite{rollander}*{Theorem 4.16} and \cite{brich}*{Theorem 5.3.2-3}, the equivalence between 
the uniform continuity principle \textbf{UC} and the fan theorem is proved, assuming that all type 2-objects are (pointwise) continuous as in \textbf{CC}.  Our version of this result is the following corollary to Theorem~\ref{trikkeagain}.
Recall the princple UFAN$_{2}$, i.e.\ the uniform version of the fan theorem from Section~\ref{qfffan}.  
\begin{thm}\label{seealso1}
In $\RCAo$, we have $[\UFAN_{2}+\eqref{POC}]\asa \eqref{MUC}$.  The same equivalence holds relative to `\st' in $\RCAO$.    
\end{thm}
\begin{proof}
The reverse direction is immediate by defining the functional $\Phi^{3}$ as $\Phi(g):=\max_{|\gamma|=\Omega(g)\wedge |\gamma|\leq_{0^{*}}1}g(\gamma*00\dots)$.  For the forward direction, fix $\varphi^{2}$ and consider the functional $\Delta$ from 
\eqref{POC}.  Then $H(\cdot):=\Delta(\varphi, \cdot)$ is also a type 2-object and consider $G(\alpha):=\Delta(H, \alpha)$.  In other words, $\Delta$ witnesses its own continuity.  
Now, in order to apply the uniform fan theorem, we have by \eqref{POC} that $(\forall \beta^{1}\leq_{1}1)(\exists n^{0})\big[\Delta(\varphi, \overline{\beta}n)\leq n\big]$ as in particular 
$(\forall \beta^{1}\leq_{1}1)\big[\Delta(\varphi, \overline{\beta}G(\beta))\leq G(\beta)\big]$.  By UFAN$_{2}$, we have $(\forall \beta^{1}\leq_{1}1)(\exists n \leq \Phi(G, T_{0}))\big[\Delta(\varphi, \overline{\beta}n)\leq n\big]$, where the tree $T_{0}$ has an obvious definition.   
Hence, if $\overline{\alpha}\Phi(G,T_{0})=\overline{\beta}\Phi(G,T_{0})$, then there is $n, m\leq \Phi(G, T_{0})$ such that $\Delta(\varphi, \overline{\beta}n)\leq n, \Delta(\varphi, \overline{\alpha}m)\leq m$.  
But then $\varphi(\alpha)=\varphi(\overline{\alpha}m*00)=\varphi(\overline{\beta}n*00)=\varphi(\beta)$ by \eqref{POC}. 
The above holds relative to `st'.  
\end{proof}
The previous theorem suggests that \eqref{POC} is the right assumption to connect the (classically acceptable by \cite{firstHORM}*{\S5}) uniform fan theorem and the (intuitionistic) fan functional.  
Perhaps surprisingly, the principle \eqref{POC} also yields equivalence between the `non-uniform' fan theorem and the uniform fan theorem as in Corollary~\ref{teks}.   
We first prove Theorem \ref{seealso999}, for which we need the following definition.
\bdefi[See \cite{kohlenbach4}*{Def.\ 4.3}]\label{kodef}
For a pointwise continuous functional $\Phi^{2}$, the sequence $\alpha^{1}$ is an \emph{associate} for $\Phi$, if they satisfy the following:
\begin{align}
(\forall f^{1})&(\exists n^{0})(\alpha(\overline{f}n)>0)\wedge \notag\\ 
&(\forall f^{1},n^{0})\big[\alpha(\overline{f}n)>0 \wedge (\forall k<n)(\alpha(\overline{f}k)=0)\di \alpha(\overline{f}n)=\Phi(f)+1 \big]. \label{myassociate}
\end{align}

\edefi

\begin{thm}\label{seealso999}
In $\RCAO$, we have $[\FAN^{\st}+\eqref{POC}]\asa \eqref{MUC}$.\\  
In $\RCAo+\QFAC^{2,0}$, we have $ [\FAN +\eqref{POC}]\asa \eqref{MUC}$. 
\end{thm}
\begin{proof}
For the first forward implication, by the proof of \cite{kohlenbach4}*{Prop.\ 4.4}, \eqref{POC} yields a (standard) functional $\Psi^{2\di 1}$ such that $\Psi(\varphi)$ is an associate of $\varphi^{2}$.  The first conjunct  of the definition of associate, namely \eqref{myassociate}, yields $(\forall^{\st} \varphi^{2})\big[(\forall^{\st} \beta^{1}\leq_{1}1)(\exists^{\st}k^{0})\Psi(\varphi)(\overline{\beta}k)>0\big]$ thanks to \eqref{POC}.    
Since the atomic formula in the former formula only takes $\overline{\beta}k$ as argument, we may apply FAN$^{\st}$ to obtain 
\[
(\forall^{\st} \varphi^{2})(\exists^{\st}n^{0})\big[(\forall^{\st} \beta^{1}\leq_{1}1)(\exists k^{0}\leq n)\Psi(\varphi)(\overline{\beta}k)>0\big], 
\]
which trivially implies 
\be\label{koel}
(\forall^{\st} \varphi^{2})(\exists^{\st}n^{0})\big[(\forall \beta^{1}\leq_{1}1)(\exists k^{0}\leq n)\Psi(\varphi)(\overline{\beta}k)>0\big].  
\ee
Now apply HAC$_{\textup{int}}$ to the previous formula to obtain standard $\Phi^{3}$ such that 
\[
(\forall^{\st} \varphi^{2})(\exists n^{0}\in \Phi(\varphi))\big[(\forall \beta^{1}\leq_{1}1)    (\exists k^{0}\leq n)\Psi(\varphi)(\overline{\beta}k)>0\big].  
\]
As usual, define $\Theta(\varphi)$ as $\max_{i<|\Phi(\varphi)|}\Phi(\varphi)(i)$;  By the second component of the definition of associate, $\Theta$ is exactly the fan functional (relative to `st').  
By Theorem~\ref{muck}, the first equivalence now follows.  
For the second forward implication, one obtains \eqref{koel} without `st' in much the same way.  This formula immediately implies:
\[
(\forall \varphi^{2})(\exists n^{0})\big[(\forall \beta^{0}\leq_{0^{*}}1)[~|\beta|=n\di (\exists k^{0}\leq n)\Psi(\varphi)(\overline{\beta}k)>0~]\big].  
\]
Now apply $\QFAC^{2,0}$ to again obtain the fan functional, and we are done.  
\end{proof}
\begin{cor}\label{teks}
In $\RCAO+\eqref{POC}$, we have $\FAN^{\st}\asa \UFAN_{2}^{\st}$.  The internal equivalence holds over $\RCAo+\QFAC^{2,0}+\eqref{POC}$.    
\end{cor}
The previous theorem expresses that the fan functional can be decomposed as the fan theorem and an intuitionistic uniform continuity principle.  
We now provide an alternative decomposition into the \emph{quantifier-free} fan theorem and a classical uniform continuity principle.  
\begin{cor}\label{seealso2}
In $\RCAo+\QFAC^{2,0}$, we have $\big[\textup{QF-FAN}+\eqref{POC}_{0}\big]\asa \eqref{MUC}$. 
\end{cor}
\begin{proof}
We only need to prove the forward implication.  By Corollary \ref{tothecenter}, QF-FAN implies that every type 2-functional is continuous and \eqref{POC} follows from \eqref{POC}$_{0}$.  
As QF-FAN implies FAN, the theorem now follows from Corollary \ref{teks}.  
\end{proof}
A natural question is whether e.g.\ $\UFAN_{2}$ plus a non-uniform version of \eqref{POC} is also equivalent to the fan functional.  
We can interpret the previous corollary as yielding $\eqref{MUC}\asa \eqref{dikjui}$, assuming \eqref{POC}$_{0}$ and $\QFAC^{2,0}$ (See Corollary~\ref{tothecenter}).  In other words, thanks to the latter princples, we may freely replace the existential quantifier in \eqref{dikjui} by a functional, along the lines of the 
central feature of Explicit Mathematics, namely that a proof of existence of an object yields a procedure to compute said object.
The following corollary expresses these results.  
\begin{cor}\label{seealso3}
In $\RCAo+\eqref{POC}+\QFAC^{2,0}$, we have $\FAN \asa \UFAN_{2}\asa \textup{QF-FAN}\asa \eqref{qffan}\asa \eqref{MUC}$.
The same equivalences hold relative to `\st'.  
\end{cor}
The previous corollary suggests that, over a weak (intuitionistic) base theory, any theorem classically equivalent to weak K\"onig's lemma is equivalent to the fan functional.  
The same seems to hold for the uniform version if the latter is \emph{constructively}\footnote{By `constructively', we mean: provable in Errett Bishop's \emph{Constructive Analysis} (\cite{bish1}).\label{bitchje}} equivalent to the fan theorem (See also the conjecture in \cite{firstHORM}*{\S3}).  
We discuss this in more detail in the next section.  

\medskip

Finally, as hinted at above, a natural question emerging from Reverse Mathematics is whether a natural mathematical theorem can be split into two natural ones, i.e.\ find natural theorems of ordinary mathematics $T, S, R$ such that $T\asa S+R$ over $\RCA_{0}$, but neither $S$ or $R$ separately implies $T$.  

\medskip

Montalb\'an discusses this question in \cite{montahue}*{p.~435} and an answer is provided in \cite{schirfeld}, 
though the former author qualifies the results regarding the splitting of Ramsey's theorem for pairs only as `somewhat natural'.  
In our opinion, the splitting results for the fan functional discussed in this section, involve truly natural principles.  

\subsection{The general case}\label{dilkooo}
In this section, we obtain the Reverse Mathematics classification of the Brouwer's continuity theorem.   

\medskip

To this end, let the \emph{Brouwer Continuity Theorem}, BCT for short, be the statement that every real function is 
uniformly continuous on $[0,1]$, i.e.\ BCT is the statement that for every $\R\di \R$-function $F$, we have 
\be\label{lakke}\textstyle
(\forall k^{0})(\exists N^{0})(\forall x^{1},y^{1}\in [0,1])(|x-y|<\frac{1}{N}\di |F(x)-F(y)|<\frac{1}{k}).
\ee  
Let UBCT be BCT with a functional $\Phi^{(1\di 1)\di 1}$ outputting the number $N$ in \eqref{lakke}.         
Furthermore, let $T$ be the statement $(b)$ from \cite{kohlenbach3}*{p.\ 293} that a (pointwise) continuous function has a supremum, i.e.\ 
\[\textstyle
(\forall F\in C[0,1])(\exists y^{1})\big[(\forall x\in [0,1])(F(x)\leq y) \wedge (\forall k^{0})(\exists z\in [0,1])(F(z)> y-\frac{1}{k})\big].   
\]
Let $UT$ be $T$ with the extra existence of a functional $\Psi^{(1\di1)\di1}$ such that $\Psi(F)$ is the supremum $y$ from $T$ if $F\in C[0,1]$.  
Finally, let $T^{*}$ be \eqref{CONT2} from Section \ref{fafi}.    
\begin{thm}\label{allisone}
In $\RCAO+\eqref{POC}+\QFAC^{2,0}$, we have 
\begin{align}\label{allisone2}
\eqref{MUC}^{\st}&\asa\eqref{dikjui}^{\st}\asa \eqref{POS}^{\st} \asa \eqref{CONT}\asa \textup{BCT}^{\st}\asa \textup{UBCT}^{\st} \asa \textup{UBT}^{*}  \notag\\
&\asa \FAN^{\st}\asa\UFAN_{2}^{\st}\asa T^{\st}\asa UT^{\st}\asa T^{*}.
\end{align}
The associated internal principles are equivalent over $\RCAo+\eqref{POC}+\QFAC^{2,0}$.  
\end{thm}
\begin{proof}
Immediate from the previous results, the Reverse Mathematics of WKL as in \cite{simpson2}*{I.10.3} and of \eqref{MUC} as in \cite{kohlenbach2}*{p.\ 293}.  
For instance, if $F$ is \emph{nonstandard} uniformly continuous as in \eqref{CONT}, it is uniformly continuous in the usual $\eps$-$\delta$-sense, yielding $\WKL$ by \cite{simpson2}*{I.10.3}.    
Now use Theorems \ref{muck} and \ref{seealso999}.  
\end{proof}
The equivalence of \eqref{MUC} and the `non-computable' principle $\WKL$ is a complement to Tait's result that the fan functional as in \cite{noortje}*{Def.\ 4.35} is `recursive(ly countable) but not computable' as proved in \cite{noortje}*{Theorems 4.36 and 4.40} and \cite{gandymahat}*{p.\ 416-417}.  
Again, by \cite{kohlenbach4}*{Cor.\ 4.9}, the assumption \eqref{POC} does not seem to be a strong one.  In other words, assuming the latter weak intuitionistic principle, a plethora of equivalences 
as in \eqref{allisone2} emerges.  

\medskip
 
As mentioned above, it seems possible to replace $T$ in \eqref{allisone2} by any theorem such that $\FAN\asa T$ constructively$^{\ref{bitchje}}$, e.g.\ concerning Riemann integration (\cite{simpson2}*{I.10.3.5}), 
polynomial approximation\footnote{In light of \cite{samzoo}*{\S3.2}, to obtain a uniform version of \cite{simpson2}*{IV.2.5} equivalent to $\WKL$, the functional should output \emph{a finite list} of polynomials, similar to HAC$_{\textup{int}}$.} (\cite{simpson2}*{IV.2.5}), and unique existence statements (\cite{ishberg}).  

\medskip

We consider the case for Riemann integration.  
Let $S$ be the statement that a continuous function is Riemann integrable on $[0,1]$, let $\US$ be $S$ with the existence of a functional $\Psi^{(1\di 1)\di 1}$ such that $\Psi(F)$ is the Riemann integral for $F\in C[0,1]$, and let $S^{*}$ be the statement that for every standard $F\in C[0,1]$, 
the Riemann sums are infinitely close for infinitesimal partitions, i.e.\ $S_{\pi}(F)\approx S_{\pi'}(F)$, for $\pi=(0,t_{1}, \dots, t_{M}, 1)$ with $\max_{i\leq M}|t_{i-1}-t_{i}|\approx 0$, and $\pi'$ similar.   
\begin{cor}\label{classicalcase}  
The equivalence \eqref{allisone2} can be extended by $\dots \asa S^{\st}\asa \US\,{^{\st}}\asa S^{*}$.  
\end{cor}
\begin{proof}
First note that the Riemann integral of a uniformly continuous function \emph{with a modulus} even exists constructively by \cite{bish1}*{p.\ 47}.  
The same holds for e.g.\ the supremum and the polynomial approximation by \cite{bish1}*{p.\ 35 and p.\ 100}.  Hence, as the fan functional \eqref{MUC} provides a modulus of uniform continuity (See \cite{kohlenbach2}*{p.~293}), it is straightforward to obtain the functional outputting the Riemann integral.  Similarly, by Corollary~\ref{tochwelbela}, we may assume nonstandard continuity, immediately yielding that such a function is nonstandard Riemann integrable as in $S^{*}$.  Clearly, both the latter and $US^{\st}$ imply $S^{\st}$, which yields $\WKL^{\st}$ by \cite{simpson2}*{IV.2.7}.  
\end{proof}
We could also study the \emph{continuous uniform boundedness principle} CUB from \cite{gako}*{\S6} in this context.   
Since $\Sigma_{0}^{0}$-CUB is equivalent to WKL and in light of its syntactic structure, it is clear that the uniform version of $\Sigma_{0}^{0}$-CUB implies $\WKL$ and follows from UFAN$_{2}$.  
Hence, it also behaves as in \eqref{allisone2}.    

\subsection{An alternative nonstandard version}\label{layola}
In this section, we suggest a slight extension of the EMT, as follows:  We formulate a nonstandard version $T^{**}$ equivalent to $UT$, for certain theorems $T$.  
The template $T^{**}$ expresses that a weak property (like pointwise continuity) holds at \emph{every} point of the space at hand, in contrast to a strong property (like uniform continuity in case of \eqref{MUC}$^{\st}$) holding at every standard point.  A first example was \eqref{donkio} in Corollary \ref{kloppip}.     
\begin{cor}  
The equivalence \eqref{allisone2} can be extended by $\dots\asa \eqref{burka}$, the latter expressing pointwise continuity at \emph{every} point of the unit interval, i.e.\  
\begin{align}\label{burka}\tag{$\mathfrak{Z}$}\textstyle
(\forall^{\st}F:[0,1]\di \R)&(\forall^{\st}k^{0})(\forall x^{1}\in [0,1])(\exists^{\st}N^{0}) \\
&\textstyle(\forall^{\st}y^{1}\in [0,1])(|x-y|<\frac{1}{N}\di |F(x)-F(y)|<\frac{1}{k}).\notag
\end{align}
\end{cor}
\begin{proof}
Similar to the proof of Corollary \ref{kloppip}, we can derive that every standard $F^{1\di 1}$ is bounded on $[0,1]$.  By \cite{simpson2}*{IV.2.3}, we obtain $\WKL^{\st}$ from \eqref{burka}.  Furthermore, the latter easily follows from the (standard) pointwise continuity of $F$ together with the nonstandard continuity as in \eqref{CONT}.    
\end{proof}
As suggested by Corollary \ref{classicalcase}, results from Friedman-Simpson Reverse Mathematics can be used to obtain equivalences as in Theorem \ref{allisone}.   
The Heine-Borel lemma constitutes another example as it is constructively equivalent to the fan theorem.  
It is straightforward to obtain the EMT and results similar to \eqref{allisone2} for the former (See e.g.\ \cite{firstHORM}*{\S5}).  
However, the Heine-Borel lemma also has an interesting formulation akin to \eqref{burka}, as in the following.     
Note that $I_{n}^{0\di (1\times 1)}$ is an open cover in that $I_{n}=(c_{n},d_{n})$ for sequences of reals $c_{n},d_{n}$ such that $(x\in I_{n})\equiv (c_{n}<x<d_{n})$.  
\begin{cor}
The equivalence \eqref{allisone2} can be extended by $\dots\asa \eqref{burka2}$, the latter stating that standardly covering $[0,1]$ implies fully covering $[0,1]$, i.e.\
\be\label{burka2}\tag{$\mathfrak{B}$}
(\forall^{\st}I_{n}^{0\di (1\times 1)})\big[(\forall^{\st}x\in [0,1])(\exists^{\st}n^{0})(x\in I_{n})\di (\forall x\in [0,1])(\exists^{\st}n^{0})(x\in I_{n}) \big]
\ee
\end{cor}
\begin{proof}
Apply \eqref{WCP} to the consequent of \eqref{burka2} to obtain a finite cover of $[0,1]$.  By \cite{simpson2}*{I.10.3}, WKL$^{\st}$ follows from \eqref{burka2}.  To obtain the latter, use the Heine-Borel lemma to obtain a finite cover for $I_{n}$ as in the antecedent of \eqref{burka2}.  It is easy to verify that this finite cover also covers the nonstandard points in $[0,1]$.  
\end{proof}
Another principle akin to \eqref{burka} and \eqref{burka2} is the following, based on \cite{simpson2}*{IV.2.3.3}. 
\begin{cor}
The equivalence \eqref{allisone2} can be extended by $\dots\asa \eqref{burka3}$, the latter expressing that a standard function $F:[0,1]\di \R$ is finite \emph{everywhere} in $[0,1]$, i.e.\
\be\label{burka3}\tag{$\mathfrak{J}$}
(\forall^{\st}F:[0,1]\di \R)(\forall x\in [0,1])(\exists^{\st}N^{0})(|F(x)|<N).
\ee
\end{cor}
\begin{proof}
Applying \eqref{WCP} to \eqref{burka3}, clearly $F$ is bounded for all standard $x\in [0,1]$, and \cite{simpson2}*{IV.2.3} yields $\WKL^{\st}$ from \eqref{burka3}.  To obtain the latter, use \eqref{CONT}.    
\end{proof}
Finally, we prove the equivalence between \eqref{SUP} and \eqref{todo} from Section \ref{fafi}.  
\begin{thm}\label{dofff}
In $\RCAO+\QFAC^{2,0}$, \eqref{SUP} is equivalent to \eqref{todo}. 
\end{thm}
\begin{proof}
Apply \eqref{WCP} to \eqref{todo}; Use Theorem \ref{halleh} and the proof of Theorem \ref{muck2}.  
\end{proof}
If Reverse Mathematics were to be about `obtaining as many equivalences as possible', \eqref{POC} would surely be a fruitful principle.  

\bye

\bye